\documentclass{article}

\usepackage{microtype}
\usepackage{graphicx}
\usepackage{subfigure}
\usepackage{booktabs} 

\usepackage{hyperref}

\usepackage[accepted]{icml2024}

\usepackage{amsmath}
\usepackage{amssymb}
\usepackage{mathtools}
\usepackage{amsthm}
\usepackage{mathrsfs}
\usepackage[capitalize,noabbrev]{cleveref}
\usepackage{stfloats}
\usepackage{bbm}
\usepackage{placeins}
\usepackage{amsmath}
\usepackage{amssymb}
\usepackage{mathtools}
\usepackage{amsthm}
\usepackage{enumerate}  
\usepackage{amsthm,amsmath,amssymb}
\usepackage{multirow} 
\usepackage{mathrsfs}
\usepackage[capitalize,noabbrev]{cleveref}
\usepackage{booktabs}

\theoremstyle{plain}
\newtheorem{theorem}{Theorem}[section]

\newtheorem{corollary}[theorem]{Corollary}
\theoremstyle{definition}
\newtheorem{definition}[theorem]{Definition}
\newtheorem{assumption}[theorem]{Assumption}
\theoremstyle{remark}

\def\argmin{\mathop{\rm arg\,min}}%

\usepackage{paralist}

\usepackage[textsize=tiny]{todonotes}

\icmltitlerunning{Differentiable Distributionally Robust Optimization Layers}

\begin{document}

\twocolumn[
\icmltitle{Differentiable Distributionally Robust Optimization Layers}



\icmlsetsymbol{corr}{$\dag$}

\begin{icmlauthorlist}
\icmlauthor{Xutao Ma}{sjtu}
\icmlauthor{Chao Ning}{sjtu,corr}
\icmlauthor{Wenli Du}{ecu}
\end{icmlauthorlist}

\icmlaffiliation{sjtu}{Department of Automation, Shanghai
Jiao Tong University, Shanghai 200240, China}
\icmlaffiliation{ecu}{The Key Laboratory of Smart Manufacturing in Energy Chemical Process, Ministry of Education, 
East China University of Science and Technology, Shanghai 200237, China}

\icmlcorrespondingauthor{Chao Ning}{chao.ning@sjtu.edu.cn}

\icmlkeywords{Machine Learning, ICML}

\vskip 0.3in
]



\printAffiliationsAndNotice{\textsuperscript{$\dag$}Corresponding Author} 

\begin{abstract}
In recent years, there has been a growing research interest in decision-focused learning, which embeds optimization problems as a layer in learning pipelines and demonstrates a superior performance than the prediction-focused approach. 
However, for distributionally robust optimization (DRO), a popular paradigm for decision-making under uncertainty, it is still unknown how to embed it as a layer, \emph{i.e.}, how to differentiate decisions with respect to an ambiguity set. 
In this paper, we develop such differentiable DRO layers for generic mixed-integer DRO problems with parameterized second-order conic ambiguity sets and discuss its extension to Wasserstein ambiguity sets.
To differentiate the mixed-integer decisions, we propose a novel dual-view methodology by handling continuous and discrete parts of decisions via different principles. Specifically, we construct a differentiable energy-based surrogate to implement the dual-view methodology and use importance sampling to estimate its gradient.
We further prove that such a surrogate enjoys the asymptotic convergency under regularization.
As an application of the proposed differentiable DRO layers, we develop a novel decision-focused learning pipeline for contextual distributionally robust decision-making tasks and compare it with the prediction-focused approach in experiments.
\end{abstract}

\section{Introduction}
In real-world scenarios, decision-making problems are typically affected by uncertainties. Therefore, machine learning techniques are usually leveraged to predict the behavior of the uncertainty, and then this prediction is passed to an optimization problem to derive decisions \cite{ning2019optimization}. Conventionally, the learning model is trained by minimizing a prediction loss, \emph{i.e.}, in a prediction-focused way.

In recent years, decision-focused learning, also known as smart predict-and-optimize in operations research \cite{elmachtoub2022smart}, has received much research interest \cite{mandi2023decision,sadana2023survey}. Different from prediction-focused learning, decision-focused learning aims to train a learning model that minimizes a decision loss, \emph{i.e.}, improving the decision quality. 
To implement decision-focused learning, differentiable optimization layers play the central role of passing gradient information from the decision back to the learning model, and this is achieved by differentiating the decision with respect to the learning target.

From the learning side, the learning target of most differentiable optimization layer research is a point prediction of uncertain quantity, and some research learns to predict the distribution of uncertainty. However, in prior research, the robustness of prediction is typically ignored, so the decision made based on this prediction is also in lack of robustness. As an emerging paradigm for robust decision-making, distributionally robust optimization (DRO) has seen a boom in both theory and applications in recent years \cite{delage2010distributionally,wiesemann2014distributionally,mohajerin2018data,rahimian2022frameworks}. Therefore, to improve decision quality while preserving robustness, developing a differentiable DRO layer to learn the ambiguity set in a decision-focused way is highly desired but has not been investigated yet.

From the optimization side, most differentiable optimization layer research focuses on either pure continuous or pure discrete decisions. However, the decisions in practical problems are typically mixed-integer. Only \citet{ferber2020mipaal} developed a mixed-integer linear program (MILP) layer. However, their approach relies on the specific solution structure of linear program (LP). Therefore, how to differentiate the mixed-integer decisions for generic mixed-integer convex optimization remains an unsolved problem.

To fill the aforementioned research gaps, this paper develops the first differentiable DRO layers with mixed-integer decisions. That is, the learning target is an ambiguity set and the output decisions are mixed-integer. The ambiguity set we mainly focus on is the class of parameterized second-order conic (SOC) ambiguity set \cite{bertsimas2019adaptive}, which is widely adopted in various applications \cite{zhou2019distributionally,zhang2022emergency,yang2023distributionally}, and Wasserstein ambiguity set is also discussed in \cref{wass_extend}.
\begin{figure*}[htpb!]
\begin{center}
\centerline{\includegraphics[width=0.8\textwidth]{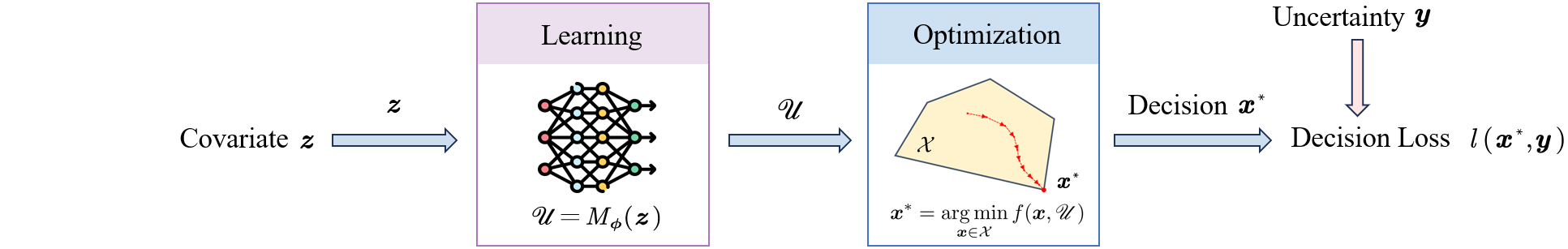}}
\caption{Sequential learning and decision-making pipeline.}
\label{decision_uncertainty}
\end{center}
\vspace{-23pt}
\end{figure*}

The major contributions of this paper are summarized as follows.
\begin{itemize}
    \item We develop the first generic differentiable DRO layers, which enable integrating learning and distributionally robust decision-making via gradient descent. 
    \item We propose a novel dual-view methodology to differentiate the mixed-integer decisions. We note that this methodology can be applied to develop any mixed-integer convex optimization layer, not limited to the proposed DRO layers.
    \item We construct a differentiable energy-based surrogate value function to implement the dual-view methodology and use importance sampling to estimate its gradient. In theory, we prove that such a surrogate enjoys the asymptotic convergency under regularization.
    \item As an application of the proposed differentiable DRO layers, we develop a novel decision-focused learning pipeline, which is of interest in its own right, for contextual distributionally robust decision-making tasks and compare it with the prediction-focused approach in experiments.
\end{itemize}

\vspace{-5pt}
\section{Related Literature}

We first review existing work on differentiable optimization layers with pure continuous and pure discrete decisions.

\textbf{Convex optimization layers.} To differentiate continuous decisions of constrained optimization, the basic idea is to apply the implicit differentiation theorem to the optimality conditions. Following this idea, \citet{amos2017optnet} successfully differentiated through constrained quadratic programs. Differentiating through LP was achieved by adding regulation terms in \citet{wilder2019melding} and \citet{mandi2020interior}. For linear conic programming, the optimality condition was derived by leveraging the homogeneous self-dual embedding technique \cite{busseti2019solution}, and then the implicit differentiation was applied \cite{Agrawal2019DifferentiatingTA}. Finally, \citet{agrawal2019differentiable} aggregated all these work and developed the differentiable convex optimization layer package \emph{cvxpylayers}.

\textbf{Combinatorial optimization layers.} To handle the non-differentiability of discrete decisions, \citet{berthet2020learning} developed differentiable surroagte solution by adding perturbation. Similar ideas also appeared in \citet{niepert2021implicit} and \citet{Pogan}. We refer to \citet{dalle2022learning} for a review of this perturbation technique.

Aside from the differentiable optimization layer approach that manages to differentiate the decision, some research constructs a surrogate loss to circumvent difficulty.

\textbf{Surrogate loss approach.} The seminal work \citet{elmachtoub2022smart} developed a surrogate $\text{SPO}^+$ loss. \citet{shah2022decision} and \citet{zharmagambetov2023landscape} constructed a training-based surrogate loss. \citet{kong2022end} developed a surrogate loss for stochastic programming (SP) by using an energy-based model. For combinatorial optimization problems, \citet{mulamba2020contrastive} and \citet{mandi2022decision} constructed surrogate loss functions by maximizing the probability of the ground-truth optimal decision.

From the perspective of the learning target, most of the work mentioned above only considered point prediction, except for \citet{donti2017task} and \citet{kong2022end}, which learned conditional distributions. \citet{chenreddy2022data} and \citet{sun2023predict} investigated prediction-focused learning methods for uncertainty sets, and \citet{wang2023learning} developed a learning method for robust optimization (RO) based on an augmented Lagrangian method. Very recently, \citet{Chenreddy2024EndtoendCR} developed an end-to-end learning method for robust optimization.

Perhaps the most relevant work to this paper is \citet{costa2023distributionally}, which to the best of our knowledge is the only research on distributionally robust decision-focused learning. However, their framework presumes the uncertainty distribution to have a residual structure and only applies to specific financial problems with continuous decisions.
On the contrary, our differentiable DRO layers apply to general distributions and a broad family of optimization problems with mixed-integer decisions.
\section{Background}

In this section, we provide some background information on the topic of this paper.
\subsection{Decision-Making under Uncertainty}
A typical sequential learning and decision-making pipeline is shown in \cref{decision_uncertainty}, where the decision-maker first leverages a learning model $M_{\boldsymbol{\phi}}$ to predict some information $\mathscr{U}$ concerning the uncertainty $\boldsymbol{y}$ from covariate $\boldsymbol{z}$. Such information $\mathscr{U}$ can be a point prediction, conditional distribution, uncertainty set, or ambiguity set of the uncertainty $\boldsymbol{y}$. 

Subsequently, the decision-maker takes $\mathscr{U}$ as a parameter and solves a constrained optimization problem $\min_{\boldsymbol{x}\in \mathcal{X}}f(\boldsymbol{x},\mathscr{U})$ to derive the decision $\boldsymbol{x}^{\ast}$. Depending on the form of $\mathscr{U}$, this constrained optimization problem can be deterministic optimization, SP, RO, or DRO. 

Finally, after the decision is made, the uncertainty $\boldsymbol{y}$ is revealed and the decision loss $l(\boldsymbol{x}^{\ast},\boldsymbol{y})$ is realized.
\subsection{Decision-Focused Learning}

In the conventional prediction-focused approach, the learning model is trained independently of the subsequent optimization process. On the contrary, in decision-focused learning, the learning model is trained by directly minimizing the decision loss, which can be formally expressed as the following bilevel problem.
\begin{equation}\label{bilevel}
    \begin{aligned}
    \min_{\boldsymbol{\phi}\in \Phi}&\ \  \hspace{26pt} \mathbb{E}_{(\boldsymbol{z},\boldsymbol{y})\sim \mathbb{P}}\ l\big(\boldsymbol{x}^{\ast}(M_{\boldsymbol{\phi}}(\boldsymbol{z})),\boldsymbol{y}\big)\\
    \text{s.t. }&\ \ \boldsymbol{x}^{\ast}(M_{\boldsymbol{\phi}}(\boldsymbol{z}))=\argmin_{\boldsymbol{x}\in \mathcal{X}} f(\boldsymbol{x},\mathscr{U}=M_{\boldsymbol{\phi}}(\boldsymbol{z}))
\end{aligned}
\end{equation}
where $\boldsymbol{\phi}$ is the parameter of the learning model $M_{\boldsymbol{\phi}}$ we want to train, $\mathbb{P}$ is the joint distribution of covariate $\boldsymbol{z}$ and uncertainty $\boldsymbol{y}$, and here we assume (\ref{bilevel}) is well-defined, \emph{i.e.}, the solution set of the argmin operator is a singleton.

To optimize this bilevel problem by gradient descent, it necessitates the computation of the following gradient.
\begin{equation}
\begin{aligned}
    &\frac{\partial l\big(\boldsymbol{x}^{\ast}(M_{\boldsymbol{\phi}}(\boldsymbol{z})),\boldsymbol{y}\big)}{\partial \boldsymbol{\phi}}\\=&\frac{\partial l\big(\boldsymbol{x}^{\ast}(M_{\boldsymbol{\phi}}(\boldsymbol{z})),\boldsymbol{y}\big)}{\partial \boldsymbol{x}^{\ast}} \frac{\partial \boldsymbol{x}^{\ast}}{\partial M_{\boldsymbol{\phi}}(\boldsymbol{z})} \frac{\partial M_{\boldsymbol{\phi}}(\boldsymbol{z})}{\partial \boldsymbol{\phi}}
\end{aligned}
\end{equation}
where the first and last terms are easy to compute. However, the existence of argmin operator poses great difficulty in the computation of the middle term $\frac{\partial \boldsymbol{x}^{\ast}}{\partial \mathscr{U}}$, so the goal of a differentiable optimization layer is to compute this term.

In this paper, we aim to develop a differentiable DRO layer, \emph{i.e.}, the learning target $\mathscr{U}$ is an ambiguity set and $\min_{\boldsymbol{x}\in \mathcal{X}} f(\boldsymbol{x},\mathscr{U})$ is a DRO problem. Therefore, the goal is to develop a method to differentiate the mixed-integer decision $\boldsymbol{x}^{\ast}$ with respect to the ambiguity set $\mathscr{U}$, \emph{i.e.}, computing $\frac{\partial \boldsymbol{x}^{\ast}}{\partial \mathscr{U}}$.

\subsection{Distributionally Robust Optimization}

The DRO takes an ambiguity set as the parameter and outputs a decision by optimizing the following problem.
\begin{equation}
    \boldsymbol{x}^{\ast}(\mathscr{U})=\argmin_{\boldsymbol{x}\in \mathcal{X}}f(\boldsymbol{x},\mathscr{U}):=  \max_{\mathbb{P}\in \mathscr{U}} \mathbb{E}_{\boldsymbol{y}\sim \mathbb{P}}[c(\boldsymbol{x},\boldsymbol{y})]\label{DRO}
\end{equation}
where the cost function $c$ is usually taken as the decision loss $l$ and $f$ is often referred to as `worst-case expectation'.

\section{Differentiable Distributionally Robust Optimization Layers}
\label{DDROL}

Since the space of all ambiguity sets is infinite-dimensional, directly learning in this space is generally computationally impossible. Therefore, we focus on the class of parameterized SOC ambiguity sets, which stem from the well-known SOC ambiguity set \cite{bertsimas2019adaptive}.

To define the parameterized SOC ambiguity set, we first introduce the following differentiable parameterized second-order cone representable set, which is an extension of the conventional second-order cone representable set (see \cref{SOC_R_S}).
\begin{definition}\label{def_soc_set}
    A set $\mathscr{W}(\boldsymbol{\theta})\subset \mathbb{R}^K$ is a differentiable parameterized second-order cone representable set with parameter $\boldsymbol{\theta}$ if there exists a collection of $J$ second-order cone inequalities such that 
    \begin{align}\nonumber
        \boldsymbol{y}\in \mathscr{W}(\boldsymbol{\theta}) 
        \Leftrightarrow \exists \boldsymbol{v}:  \boldsymbol{A}_j(\boldsymbol{\theta})\left[ \begin{gathered}
        \boldsymbol{y}\\ \boldsymbol{v}
    \end{gathered} \right]-\boldsymbol{b}_j(\boldsymbol{\theta}) \geq_{L^{m_j}} \boldsymbol{0},\forall j \in [J]
    \end{align}
    where $L^{m_j}$ represents a $m_j$ dimensional second-order cone and matrixes $\boldsymbol{A}_j(\boldsymbol{\theta})$ and vectors $\boldsymbol{b}_j(\boldsymbol{\theta})$ are differentiable functions of $\boldsymbol{\theta}$.
\end{definition}

Now we define the parameterized SOC ambiguity set.
\begin{definition}\label{def_soc}
    An ambiguity set $\mathscr{U}(\boldsymbol{\theta})$ is a parameterized SOC ambiguity set with parameter $\boldsymbol{\theta}$ if it can be expressed as follows.
    \begin{equation}\label{DEF_SOC_AMB}
        \mathscr{U}(\boldsymbol{\theta})=\left\{ \mathbb{P} \left|\ \begin{gathered}
        \mathbb{P}(\varXi)=1\\
            \mathbb{E}_{\mathbb{P}} [g_i(\boldsymbol{y},\boldsymbol{\alpha}_i)]\leq \sigma_i,\forall i\in[I]
        \end{gathered}
         \right.       \right\}
    \end{equation}
    where $\mathbb{P}$ is a distribution of $\boldsymbol{y}$, $\boldsymbol{\theta}=(\boldsymbol{\alpha}_1,\sigma_1,\cdots,\boldsymbol{\alpha}_I,\sigma_I)$, support $\varXi\subset \mathbb{R}^K$ of the uncertainty is a second-order cone representable set, and the epigraph of each $g_i$,
    \begin{equation}
        \text{epi}\ g_i =\{ (\boldsymbol{y},u) | u\geq g_i(\boldsymbol{y},\boldsymbol{\alpha}_i)  \}
    \end{equation}
    is a differentiable parameterized second-order cone representable set with parameter $\boldsymbol{\alpha}_i$
\end{definition}

By selecting functions $g_i$, the parameterized SOC ambiguity set can characterize a variety of distributional features. We present some examples of the parameterized SOC ambiguity set in \cref{exam_SOC_amb} and see also \citet{bertsimas2019adaptive}.

For the cost function $c(\boldsymbol{x},\boldsymbol{y})$ in \cref{DRO}, we consider both the single- and two-stage cost functions.
\begin{assumption}\label{ass_cost_func}
    The cost function $c(\boldsymbol{x},\boldsymbol{y})$ admits a single-stage formulation (i) without recourse or a two-stage formulation (ii) with relatively complete recourse as follows.

    (i). Single-stage formulation: $c(\boldsymbol{x},\boldsymbol{y})=\sum_{k=1}^{K} c_k(\boldsymbol{x})y_k$,
    where the epigraph of each function $c_k$ is a second-order cone representable set and the uncertainty $\boldsymbol{y}$ is required to be non-negative.

   (ii). Two-stage formulation: $c(\boldsymbol{x},\boldsymbol{y})$ is the optimal value of an LP, \emph{i.e.},
    \begin{equation}\label{cost_func}
            c(\boldsymbol{x},\boldsymbol{y})=\min_{\boldsymbol{\gamma}\geq0}\boldsymbol{q}^T \boldsymbol{\gamma}\
            \text{s.t. } \boldsymbol{T}(\boldsymbol{y}) \boldsymbol{x} +\boldsymbol{W\gamma}=h(\boldsymbol{y})
    \end{equation}
    where $\boldsymbol{T}(\boldsymbol{y})=\boldsymbol{T}_0+\sum_{k=1}^K \boldsymbol{T}_k y_k$, $h(\boldsymbol{y})=\boldsymbol{h}_0+\sum_{k=1}^K \boldsymbol{h}_k y_k$, and the constraint in (\ref{cost_func}) is feasible for all $\boldsymbol{x}\in \mathcal{X}$ and $\boldsymbol{y}\in \varXi$.
\end{assumption}

To derive a tractable reformulation of the DRO problem (\ref{DRO}), we need the following regularity assumption on the parameterized SOC ambiguity set, and the detailed explanation of \cref{ass_sla} is presented in \cref{Slater}.
\begin{assumption}
\label{ass_sla}
    Slater’s condition holds for the parameterized SOC ambiguity set $\mathscr{U}(\boldsymbol{\theta})$.
\end{assumption}

Now we can state the following reformulation theorem.
\begin{theorem}\label{thm_refor}
    Suppose $\mathscr{U}(\boldsymbol{\theta})$ is a parameterized SOC ambiguity set and \cref{ass_cost_func} and \cref{ass_sla} hold, then the worst-case expectation $f(\boldsymbol{x},\mathscr{U}(\boldsymbol{\theta}))=\max_{\mathbb{P}\in \mathscr{U}(\boldsymbol{\theta})} \mathbb{E}_{\boldsymbol{y}\sim \mathbb{P}}[c(\boldsymbol{x},\boldsymbol{y})]$ is a linear second-order cone program.
\end{theorem}
\begin{proof}
    See \cref{proof_refor}
\end{proof}
\vspace{-10pt}
Since the parameterized SOC ambiguity set is determined by the finite-dimensional parameter $\boldsymbol{\theta}$, it suffices to use the learning model $M_{\boldsymbol{\phi}}(\boldsymbol{z})$ to learn the parameter $\boldsymbol{\theta}$, \emph{i.e.},
\begin{equation}
    \boldsymbol{\theta}=M_{\boldsymbol{\phi}}(\boldsymbol{z}).
\end{equation}
Therefore, the goal of the differentiable DRO layer comes down to computing $\frac{\partial \boldsymbol{x}^{\ast}}{\partial \boldsymbol{\theta}}$.

By \cref{thm_refor}, the worst-case expectation function $f$ is a linear second-order cone programming. Therefore, if the decision $\boldsymbol{x}$ is continuous, the gradient $\frac{\partial \boldsymbol{x}^{\ast}}{\partial \boldsymbol{\theta}}$ can be directly computed by the technique of differentiating through a cone program \cite{Agrawal2019DifferentiatingTA}. We formally state this result in \cref{diff_con}.

\begin{theorem}
    \label{diff_con}
    Suppose conditions in \cref{thm_refor} hold and $\boldsymbol{x}$ is a continuous variable with $\mathcal{X}$ a second-order cone representable set, then $\boldsymbol{x}^{\ast}=\argmin_{\boldsymbol{x}\in\mathcal{X}} f(\boldsymbol{x},\mathscr{U}(\boldsymbol{\theta}))$ is differentiable with respect to $\boldsymbol{\theta}$.
\end{theorem}
\vspace{-5pt}
\begin{proof}
    See \cref{proof_diff_con}
\end{proof}
\vspace{-10pt}
However, when the decision $\boldsymbol{x}$ is mixed-integer, it is inherently non-differentiable due to the discreteness of integer variables. Therefore, it is necessary to develop a new methodology to handle mixed-integer decisions.

\subsection{Dual-View of Mixed-Integer Decisions}

Existing research on differentiable optimization layers all views the decision $\boldsymbol{x}^{\ast}$ as a function of the parameter and handles it by the principle of automatic differentiation. However, this view does not work for discrete decisions. Although some research manages to differentiate the discrete decisions by adding perturbations \cite{berthet2020learning}, these approaches are still restricted to integer linear programming. 

Alternatively, by examining the bilevel formulation (\ref{bilevel}) of the whole decision-focused learning task, we notice that the decision-making process is the lower-level problem. Therefore, the decision $\boldsymbol{x}^{\ast}$ can be viewed as a constraint, and we can handle it via the principle of constrained optimization.

By the above observations, we propose the following dual-view methodology to address mixed-integer decisions.

\textbf{Dual-View Methodology:}
\begin{enumerate}[I]
    \item The continuous part of the decisions is viewed as a function of parameters and handled via the principle of automatic differentiation.\label{dual_con}
    \item The discrete part of the decisions is viewed as a constraint of the whole bilevel learning problem and handled via the principle of constrained optimization.\label{dual_int}
\end{enumerate}

In this dual-view methodology, we have already established part \ref{dual_con} in \cref{diff_con}. To better illustrate the idea of part \ref{dual_int}, we make the following assumptions and notations.

\begin{assumption}\label{assum_mixed_int}
    $\boldsymbol{x}=(\boldsymbol{x}_d,\boldsymbol{x}_c)$ is a mixed-integer variable with discrete part $\boldsymbol{x}_d\in \{ 0,1 \}^{n_1}$ and continuous part $\boldsymbol{x}_c\in \mathbb{R}^{n_2}$. The feasible region of $\boldsymbol{x}$ is $\mathcal{X}=\overline{\mathcal{X}}\cap (\{ 0,1 \}^{n_1} \otimes \mathbb{R}^{n_2})$, where $\overline{\mathcal{X}}$ is a second-order cone representable set.
\end{assumption}

We denote by $\mathcal{X}_d$ the feasible region of the discrete part of variable $\boldsymbol{x}$, \emph{i.e.}, $\mathcal{X}_d=\{ \boldsymbol{x}_d \in  \{ 0,1 \}^{n_1} | \exists \boldsymbol{x}_c \in  \mathbb{R}^{n_2}: (\boldsymbol{x}_d,\boldsymbol{x}_c) \in \mathcal{X}  \}$, and by $\mathcal{X}_c(\boldsymbol{x}_d)$ the feasible region of the continuous part variable $\boldsymbol{x}_c$ given the integer part $\boldsymbol{x}_d\in\mathcal{X}_d$, \emph{i.e.}, $\mathcal{X}_c(\boldsymbol{x}_d)=\{ \boldsymbol{x}_c \in  \mathbb{R}^{n_2} |  (\boldsymbol{x}_d,\boldsymbol{x}_c) \in \mathcal{X} \}$.

The next assumption ensures that the bilevel problem (\ref{bilevel}) is well-defined, and see \cref{app_diss_unique} for a detailed discussion of this assumption.

\begin{assumption}\label{assum_unique_sol}
    (i) For all $\boldsymbol{\phi}\in \Phi$, $\boldsymbol{z}\in \mathcal{Z}$, and $\boldsymbol{x}_d\in \mathcal{X}_d$, the optimal continuous solution 
    \begin{equation}\nonumber
        \boldsymbol{x}_c^{\ast}(\boldsymbol{x}_d,M_{\boldsymbol{\phi}}(\boldsymbol{z})):=\argmin_{\boldsymbol{x}_c\in \mathcal{X}_c(\boldsymbol{x}_d)} f\big((\boldsymbol{x}_d,\boldsymbol{x}_c),M_{\boldsymbol{\phi}}(\boldsymbol{z})\big)
    \end{equation}
    is unique.

    (ii) For all $\boldsymbol{\phi}\in \Phi$, the optimal integer solution 
    \begin{equation}\nonumber
        \boldsymbol{x}_d^{\ast}(M_{\boldsymbol{\phi}}(\boldsymbol{z})):= \argmin_{\boldsymbol{x}_d\in \mathcal{X}_d} f\big((\boldsymbol{x}_d,\boldsymbol{x}_c^{\ast}(\boldsymbol{x}_d,M_{\boldsymbol{\phi}}(\boldsymbol{z})),M_{\boldsymbol{\phi}}(\boldsymbol{z})\big)
    \end{equation}
    is unique almost surely, \emph{i.e.},
    \begin{equation}\nonumber
        \forall \boldsymbol{\phi}\in \Phi, \mathbb{P}_{\boldsymbol{z}}\big( \boldsymbol{x}_d^{\ast}(M_{\boldsymbol{\phi}}(\boldsymbol{z})) \text{ is a singleton} \big)=1
    \end{equation}
    where $\mathbb{P}_{\boldsymbol{z}}$ is the marginal distribution of covariate $\boldsymbol{z}$.
\end{assumption}

By the above assumptions, the bilevel problem (\ref{bilevel}) can be reformulated as follows.
\begin{corollary}
    Suppose \cref{assum_mixed_int} and \cref{assum_unique_sol} hold, the decision-focused learning can be formulated as
    \begin{equation}\label{bilevel_refor}
        \begin{gathered}
            \min_{\boldsymbol{\phi}\in \Phi}\ \  \mathbb{E}_{(\boldsymbol{z},\boldsymbol{y})\sim \mathbb{P}}\ l\Big(\big(\boldsymbol{x}_d^{\ast}(M_{\boldsymbol{\phi}}(\boldsymbol{z})),\boldsymbol{x}_c^{\ast}(\boldsymbol{x}_d^{\ast},M_{\boldsymbol{\phi}}(\boldsymbol{z}))\big),\boldsymbol{y}\Big)\\
    \text{s.t. } \boldsymbol{x}_d^{\ast}(M_{\boldsymbol{\phi}}(\boldsymbol{z}))\hspace{150pt} \\ \hspace{30pt}=\argmin_{\boldsymbol{x}_d\in \mathcal{X}_d} f\big((\boldsymbol{x}_d,\boldsymbol{x}_c^{\ast}(\boldsymbol{x}_d,M_{\boldsymbol{\phi}}(\boldsymbol{z}))),M_{\boldsymbol{\phi}}(\boldsymbol{z})\big)
        \end{gathered}
    \end{equation}
\end{corollary}
From the bilevel form (\ref{bilevel_refor}) we can see more clearly the idea of dual-view methodology. The optimal continuous decision $\boldsymbol{x}_c^{\ast}$ is embedded as a function of the integer decision and learning target $M_{\boldsymbol{\phi}}(\boldsymbol{z})$. On the contrary, the optimal integer solution $\boldsymbol{x}_d^{\ast}$ is explicitly expressed as a constraint.

Let $\mathcal{R}(\boldsymbol{\phi})$ denote the value function of problem (\ref{bilevel_refor}), \emph{i.e.},
\begin{equation}\nonumber
    \begin{gathered}
       \mathcal{R}(\boldsymbol{\phi}):= \mathbb{E}_{(\boldsymbol{z},\boldsymbol{y})\sim \mathbb{P}}\ l\Big(\big(\boldsymbol{x}_d^{\ast}(M_{\boldsymbol{\phi}}(\boldsymbol{z})),\boldsymbol{x}_c^{\ast}(\boldsymbol{x}_d^{\ast},M_{\boldsymbol{\phi}}(\boldsymbol{z}))\big),\boldsymbol{y}\Big)
    \end{gathered}
\end{equation}
Then problem (\ref{bilevel_refor}) is equivalent to $\min_{\boldsymbol{\phi} \in \Phi}\mathcal{R}(\boldsymbol{\phi})$.

Following Part \ref{dual_int} in the dual-view methodology, we handle the constrained optimization (\ref{bilevel_refor}) by approximating the value function sequentially. That is,
we want to construct a sequence of differentiable surrogate value functions $\mathcal{R}_{\lambda}(\boldsymbol{\phi})$ such that $\mathcal{R}_{\lambda}(\boldsymbol{\phi})\to \mathcal{R}(\boldsymbol{\phi})$ in some sense.
\subsection{Energy-Based Surrogate Value Function}

To construct such a surrogate function $\mathcal{R}_{\lambda}(\boldsymbol{\phi})$, we first construct point surrogate function $r_{\lambda}(M_{\boldsymbol{\phi}}(\boldsymbol{z}),\boldsymbol{y})$ for each point $(\boldsymbol{z},\boldsymbol{y})$ and then define $\mathcal{R}_{\lambda}(\boldsymbol{\phi})$ as
\begin{equation}
    \mathcal{R}_{\lambda}(\boldsymbol{\phi})=\mathbb{E}_{(\boldsymbol{z},\boldsymbol{y})\sim \mathbb{P}} \ r_{\lambda}(M_{\boldsymbol{\phi}}(\boldsymbol{z}),\boldsymbol{y})
\end{equation}

We construct the point surrogate function $r_{\lambda}(\boldsymbol{\phi},\boldsymbol{z},\boldsymbol{y})$ by leveraging the energy-based model.
Specifically, we assign each feasible integer decision $\boldsymbol{x}_d\in \mathcal{X}_d$ the following energy function $E(\boldsymbol{x}_d,M_{\boldsymbol{\phi}}(\boldsymbol{z}),\lambda)$,
where $\lambda$ is a positive scalar.
\begin{equation}\label{eneger_defnm}
\begin{aligned}
    E(\boldsymbol{x}_d,&M_{\boldsymbol{\phi}}(\boldsymbol{z}),\lambda)\\=&\text{exp}\left( -\frac{f((\boldsymbol{x}_d,\boldsymbol{x}_c^{\ast}(\boldsymbol{x}_d,M_{\boldsymbol{\phi}}(\boldsymbol{z})),M_{\boldsymbol{\phi}}(\boldsymbol{z}))}{\lambda}  \right)
\end{aligned}
\end{equation}

Based on the energy function, we can define a distribution $p(\boldsymbol{x}_d|M_{\boldsymbol{\phi}}(\boldsymbol{z}),\lambda)$ over $\mathcal{X}_d$.
\begin{equation}\label{energy_distri}
    p(\boldsymbol{x}_d|M_{\boldsymbol{\phi}}(\boldsymbol{z}),\lambda)=\frac{E(\boldsymbol{x}_d,M_{\boldsymbol{\phi}}(\boldsymbol{z}),\lambda)}{\sum_{\boldsymbol{x}_d^{'}\in \mathcal{X}_d} E(\boldsymbol{x}_d^{'},M_{\boldsymbol{\phi}}(\boldsymbol{z}),\lambda)}
\end{equation}

We then construct point surrogate function $r_{\lambda}(\boldsymbol{\phi},\boldsymbol{z},\boldsymbol{y})$ 
as follows.\allowdisplaybreaks
\begin{align}
    &r_{\lambda}(M_{\boldsymbol{\phi}}(\boldsymbol{z}),\boldsymbol{y})\nonumber\\&:
=\mathbb{E}_{\boldsymbol{x}_d\sim p(\boldsymbol{x}_d|M_{\boldsymbol{\phi}}(\boldsymbol{z}),\lambda)}l((\boldsymbol{x}_d,\boldsymbol{x}_c^{\ast}(\boldsymbol{x}_d,M_{\boldsymbol{\phi}}(\boldsymbol{z}))),\boldsymbol{y})\\
    &=\sum_{\boldsymbol{x}_d\in \mathcal{X}_d}p(\boldsymbol{x}_d|M_{\boldsymbol{\phi}}(\boldsymbol{z}),\lambda) l((\boldsymbol{x}_d,\boldsymbol{x}_c^{\ast}(\boldsymbol{x}_d,M_{\boldsymbol{\phi}}(\boldsymbol{z}))),\boldsymbol{y})\nonumber
\end{align}    

By the above construction, we notice that the optimal integer solution $\boldsymbol{x}_d^{\ast}(M_{\boldsymbol{\phi}}(\boldsymbol{z}))$ has the largest energy, so the corresponding probability $p(\boldsymbol{x}_d^{\ast}|M_{\boldsymbol{\phi}}(\boldsymbol{z}),\lambda)$ is also the highest. When $\lambda\to 0^{+}$, this probability will converge to 1, and $r_{\lambda}(\boldsymbol{\phi},\boldsymbol{z},\boldsymbol{y})$ will also converge to the true decision loss
$l\big((\boldsymbol{x}_d^{\ast}(M_{\boldsymbol{\phi}}(\boldsymbol{z})),\boldsymbol{x}_c^{\ast}(\boldsymbol{x}_d^{\ast},M_{\boldsymbol{\phi}}(\boldsymbol{z}))),\boldsymbol{y}\big)$.

To further establish convergence results of the surrogate value function $\mathcal{R}_{\lambda}(\boldsymbol{\phi})$, we need the following continuity assumptions and the concept of epi-convergence.
\begin{assumption}\label{assum_bound}
    The decision loss $l((\boldsymbol{x}_d,\boldsymbol{x}_c),\boldsymbol{y})$ is bounded and continuous in $\boldsymbol{x}_c$. For any $\boldsymbol{z}\in \mathcal{Z}$, $f((\boldsymbol{x}_d,\boldsymbol{x}_c),M_{\boldsymbol{\phi}}(\boldsymbol{z}))$ is continuous in $\boldsymbol{x}_c$ and $\boldsymbol{\phi}$. 
\end{assumption}
\begin{assumption}\label{assum_opt_con}
    For all $\boldsymbol{x}_d\in \mathcal{X}_d$ and $\boldsymbol{z}\in \mathcal{Z}$, the optimal continuous decision $\boldsymbol{x}_c^{\ast}(\boldsymbol{x}_d,M_{\boldsymbol{\phi}}(\boldsymbol{z}))$ is continuous in $\boldsymbol{\phi}$.
\end{assumption}
We note that \cref{assum_bound} is easily satisfied. For \cref{assum_opt_con}, since  $\boldsymbol{x}_c^{\ast}(\boldsymbol{x}_d,M_{\boldsymbol{\phi}}(\boldsymbol{z}))$ is differentiable with respect to $\boldsymbol{\theta}=M_{\boldsymbol{\phi}}(\boldsymbol{z})$ by \cref{diff_con}, \cref{assum_opt_con} simply requires the continuity of $M_{\boldsymbol{\phi}}$ in its parameter $\boldsymbol{\phi}$.

\begin{definition}
    [\citet{bonnans2013perturbation}, p.41] A sequence of functions $\{\mathcal{R}_{n}(\boldsymbol{\phi})\}$ epi-converges to a function $\mathcal{R}(\boldsymbol{\phi})$ if and only if $\forall \boldsymbol{\phi} \in \Phi$, condition (i) and (ii) hold.
\begin{enumerate}[(i)]
    \item For any sequence $\{ \boldsymbol{\phi}_n \}$ converges to $\boldsymbol{\phi}$, $\liminf_{n\to \infty} \mathcal{R}_{n}(\boldsymbol{\phi}_n)\geq  \mathcal{R}(\boldsymbol{\phi})$
    \item There exists a sequence $\{ \boldsymbol{\phi}_n \}$ converging to $\boldsymbol{\phi}$ such that $\limsup_{n\to \infty}\mathcal{R}_{n}(\boldsymbol{\phi}_n)\leq  \mathcal{R}(\boldsymbol{\phi})$
\end{enumerate}
\end{definition}

The epi-convergence of $\mathcal{R}_{\lambda}(\boldsymbol{\phi})$ and asymptotic convergence of optimal solution are established in \cref{thm_epi_converge}.
\begin{theorem}\label{thm_epi_converge}
\allowdisplaybreaks
    Suppose \cref{assum_mixed_int}, \cref{assum_unique_sol}, \cref{assum_bound}, and \cref{assum_opt_con} hold, then for any sequence $\lambda_n \searrow 0^+$ as $n \to \infty$,  the following two
assertions hold for the energy-based surrogate value function $\mathcal{R}_{\lambda}(\boldsymbol{\phi})$.

(i) $\mathcal{R}_{\lambda_n}(\boldsymbol{\phi})$ epi-converges to $\mathcal{R}(\boldsymbol{\phi})$ as $n \to \infty$.

(ii) if $\boldsymbol{\phi}_{\lambda_{n_k}}\in \argmin_{\boldsymbol{\phi}\in \Phi} \mathcal{R}_{\lambda_{n_k}}(\boldsymbol{\phi})$ for some sub-sequence $\{ n_k \}\subset \mathbb{N}$ and $\{ \boldsymbol{\phi}_{\lambda_{n_k}}\}$ converges to a point $\boldsymbol{\phi}^{\ast}$, then $\boldsymbol{\phi}^{\ast}\in \argmin_{\boldsymbol{\phi}\in \Phi} \mathcal{R}(\boldsymbol{\phi})$ and $\lim_{k\to \infty}\inf_{\boldsymbol{\phi}\in \Phi}\mathcal{R}_{\lambda_{n_k}}(\boldsymbol{\phi})=\inf_{\boldsymbol{\phi}\in \Phi}\mathcal{R}(\boldsymbol{\phi})$
\end{theorem}

\begin{proof}
    See \cref{proof_thm_epi_converge}
\end{proof}
\vspace{-15pt}

\begin{figure*}[btp]
\begin{center}
\centerline{\includegraphics[width=\textwidth]{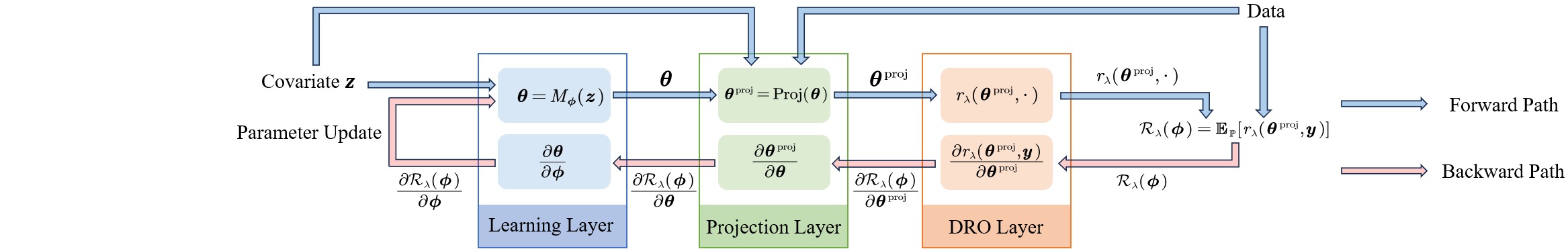}}
\caption{Decision-focused learning pipeline for contextual distributionally robust decision-making.}
\label{DF_learning}
\end{center}
\vspace{-20pt}
\end{figure*}
\subsection{Gradient Estimation}

According to \cref{thm_epi_converge}, the optimal parameter $\boldsymbol{\phi}$ of the learning model $M_{\boldsymbol{\phi}}$ can be derived by optimizing the surrogate value functions $\mathcal{R}_{\lambda}(\boldsymbol{\phi})$ sequentially. We next show in \cref{thm_diff_mixed} that the surrogate value function $\mathcal{R}_{\lambda}(\boldsymbol{\phi})$ is differentiable, so it can be optimized via gradient descent.

\begin{theorem}
    \label{thm_diff_mixed}
    Suppose conditions in \cref{thm_refor} and \cref{thm_epi_converge} hold, then $\mathcal{R}_{\lambda}(\boldsymbol{\phi})$ is differentiable with respect to $\boldsymbol{\phi}$ and the gradient is
    \begin{equation}
        \frac{\partial \mathcal{R}_{\lambda}(\boldsymbol{\phi})}{\partial \boldsymbol{\phi}} =\mathbb{E}_{(\boldsymbol{z},\boldsymbol{y})\sim \mathbb{P}}\bigg[ \frac{\partial r_{\lambda}(\boldsymbol{\theta},\boldsymbol{y})}{\partial \boldsymbol{\theta}} \frac{\partial \boldsymbol{\theta}}{\partial \boldsymbol{\phi}}\bigg]
    \end{equation}
    where $\boldsymbol{\theta}=M_{\boldsymbol{\phi}}(\boldsymbol{z})$ is the learning target, and $\frac{\partial r_{\lambda}(\boldsymbol{\theta},\boldsymbol{y})}{\partial \boldsymbol{\theta}}$ can be computed by\allowdisplaybreaks
    \begin{equation}\label{expect_grad} \allowdisplaybreaks
        \begin{aligned}
            &\  \ \ \frac{\partial r_{\lambda}(\boldsymbol{\theta},\boldsymbol{y})}{\partial \boldsymbol{\theta}}\\&=
            \mathbb{E}_{\boldsymbol{x_d}\sim p}  \Bigg[ \frac{E^{'} (\boldsymbol{x_d},\boldsymbol{\theta},\lambda)}{E (\boldsymbol{x_d},\boldsymbol{\theta},\lambda)}l\Big(\big(\boldsymbol{x_d},\boldsymbol{x_c}^{\ast}(\boldsymbol{x_d},\boldsymbol{\theta})\big),\boldsymbol{y}\Big)\Bigg]- \\&\mathbb{E}_{\boldsymbol{x_d}\sim p}\bigg[  \frac{E^{'} (\boldsymbol{x_d},\boldsymbol{\theta},\lambda)}{E (\boldsymbol{x_d},\boldsymbol{\theta},\lambda)}  \bigg]  \mathbb{E}_{\boldsymbol{x_d}\sim p}\bigg[ l\Big(\big(\boldsymbol{x_d},\boldsymbol{x_c}^{\ast}(\boldsymbol{x_d},\boldsymbol{\theta})\big),\boldsymbol{y}\Big)  \bigg]\\&
    +\mathbb{E}_{\boldsymbol{x_d}\sim p}\Bigg[ \frac{\partial l\Big(\big(\boldsymbol{x_d},\boldsymbol{x_c}^{\ast}(\boldsymbol{x_d},\boldsymbol{\theta})\big),\boldsymbol{y}\Big)}{\partial \boldsymbol{x_c}^{\ast}} \frac{\partial  \boldsymbol{x_c}^{\ast}(\boldsymbol{x_d},\boldsymbol{\theta})}{\partial \boldsymbol{\theta}} \Bigg]
        \end{aligned}
    \end{equation}
    where $p=p(\boldsymbol{x}_d|\boldsymbol{\theta},\lambda)$ and $E^{'} (\boldsymbol{x_d},\boldsymbol{\theta},\lambda)=\frac{\partial E (\boldsymbol{x_d},\boldsymbol{\theta},\lambda)}{\partial \boldsymbol{\theta}}$.
\end{theorem}
\begin{proof}
    See \cref{proof_diff_mixed}
\end{proof}
Note that in the last term of \cref{expect_grad}, $\frac{\partial  \boldsymbol{x_c}^{\ast}(\boldsymbol{x_d},\boldsymbol{\theta})}{\partial \boldsymbol{\theta}}$ is the gradient of continuous decision with respect to the parameter, which is exactly what we develop in \cref{diff_con}.

The gradient (\ref{expect_grad}) can be estimated by sampling from distribution $p(\boldsymbol{x}_d|\boldsymbol{\theta},\lambda)$. However, direct sampling from $p(\boldsymbol{x}_d|\boldsymbol{\theta},\lambda)$ necessitates the computation of the normalizer in \cref{energy_distri}, which requires the calculation of the energy function of all the feasible integer solutions.

To avoid this problem, we adopt the self-normalized importance sampling method (See \cref{app_imp_samp}). To construct a proposal distribution $q$ that resembles $p(\boldsymbol{x}_d|\boldsymbol{\theta},\lambda)$, we first derive $T$ integer solutions $\mathscr{C}=\{ \boldsymbol{x_d}^1,\cdots,\boldsymbol{x_d}^T \}$ with the largest energy functions by solving $f$ for $T$ times (See \cref{inte_cut} for readers not familiar with this oracle) and then construct the proposal distribution $q$ as follows.\allowdisplaybreaks
\begin{align}\label{impr_asmp}
    q(\boldsymbol{x_d})= 
        \frac{E(\boldsymbol{x_d},\boldsymbol{\theta},\lambda)}{\sum_{t\in [T]}E(\boldsymbol{x_d}^t,\boldsymbol{\theta},\lambda)+M(|\mathcal{X}_d|-T) },\forall \boldsymbol{x_d}\in \mathscr{C}\nonumber\\ 
        q(\boldsymbol{x_d})=\frac{M}{\sum_{t\in [T]}E(\boldsymbol{x_d}^t,\boldsymbol{\theta},\lambda)+M(|\mathcal{X}_d|-T) },\forall \boldsymbol{x_d}\notin \mathscr{C}
\end{align}
where $M$ is a constant that can be understood as the energy of other integer solutions.

Therefore, each term in \cref{expect_grad} can be estimated unbiasedly by sampling from $q$.
\section{Application: Contextual Distributionally Robust Decision-Making}
\label{C_DRO}

As an application of the differentiable DRO layers in \cref{DDROL}, we develop a decision-focused learning pipeline for contextual distributionally robust decision-making tasks \cite{bertsimas2022bootstrap,wang2021distributionally,yang2022decision}. 

In this paper, we mainly focus on and develop a decision-focused learning method for DRO with SOC ambiguity set, but in fact, the proposed DRO Layer technique can also be extended to the Wasserstein ambiguity set and we discuss this issue in \cref{wass_extend}.

\subsection{Decision-Focused Learning Pipeline}
\label{df_pipe}

The proposed pipeline is illustrated in \cref{DF_learning}. A learning model $M_{\boldsymbol{\phi}}$ is first leveraged to learn the ambiguity set parameter $\boldsymbol{\theta}$ from covariate $\boldsymbol{z}$. However, the output parameter $\boldsymbol{\theta}$ provided by the learning model can lead to an empty ambiguity set, \emph{i.e.}, $\mathscr{U}(\boldsymbol{\theta})=\varnothing
$, and this problem typically happens when the learning model is a neural network (NN). 

To fix this problem, we add a projection layer after the learning layer. The projection layer takes $\boldsymbol{\theta}$ as input and outputs $\boldsymbol{\theta}^{\text{proj}}$ such that $\mathscr{U}(\boldsymbol{\theta}^{\text{proj}})$ is always non-empty. To achieve this, we construct the projection layer as follows.
\begin{align}
    \boldsymbol{\theta}^{\text{proj}}:=&\argmin_{\boldsymbol{\theta}^{\text{proj}}} \left\| \boldsymbol{\theta}^{\text{proj}}-\boldsymbol{\theta}   \right\| \text{ s.t. }  \mathbb{Q}_{\boldsymbol{z}}\in \mathscr{U}(\boldsymbol{\theta}^{\text{proj}})\label{proj_cons}
\end{align}
In the constraint of (\ref{proj_cons}), we explicitly require that a distribution $\mathbb{Q}_{\boldsymbol{z}}$ lies in the parameterized SOC ambiguity set $\mathscr{U}(\boldsymbol{\theta}^{\text{proj}})$, which ensures the non-emptyness of $\mathscr{U}(\boldsymbol{\theta}^{\text{proj}})$. This distribution $\mathbb{Q}_{\boldsymbol{z}}$ should be understood as an estimation of the conditional distribution of uncertainty $\boldsymbol{y}$ given covariate $\boldsymbol{z}$.

To construct such a conditional distribution estimation $\mathbb{Q}_{\boldsymbol{z}}$, we take the idea in \citet{bertsimas2020predictive}, which constructed the conditional distribution from data $(\boldsymbol{z}_n,\boldsymbol{y}_n),n\in [N]$ in a weighted sample average way as follows.
\begin{equation}\label{W_SAA}
    \mathbb{Q}_{\boldsymbol{z}}=\sum_{n=1}^N \omega_{n}(\boldsymbol{z})\delta_{\boldsymbol{y}_n}, \omega_{n}(\boldsymbol{z})\geq 0, \sum_{n=1}^N \omega_{n}(\boldsymbol{z})=1
\end{equation}
where 
$\delta_{[\cdot]}$ is the Dirac delta function.

In (\ref{W_SAA}), the weight $\omega_{n}(\boldsymbol{z})$ can be intuitively understood as a measurement of closeness between $\boldsymbol{z}$ and data $\boldsymbol{z}_n$. Some research papers provide such weight functions to choose from \cite{bertsimas2020predictive,kallus2023stochastic}, for example, the \emph{k}-nearest-neighbors weight function.

With the formulation (\ref{W_SAA}) of $\mathbb{Q}_{\boldsymbol{z}}$, the constraint of (\ref{proj_cons}) is equivalent to
\begin{equation}\label{proj_cons_refor}
    \sum_{n=1}^N \omega_{n}(\boldsymbol{z}) g_i(\boldsymbol{y}_n,\boldsymbol{\alpha}_i)\leq \sigma_i, \ \forall i \in [I]
\end{equation}

Further, if $g_i,i\in [I]$ are selected as in \cref{exam_SOC_amb}, then (\ref{proj_cons_refor}) can be reformulated into finitely many second-order cone constraints (See \cref{ana_pro} for this result). 
Therefore, the projection layer is a convex optimization layer.

After projection, $\boldsymbol{\theta}^{\text{proj}}$ is fed into the DRO layer to construct the surrogate value function $\mathcal{R}_{\lambda}(\boldsymbol{\phi})$, which is usually estimated by data of a certain batch size. Then, the back-propagation and parameter update processes are conducted.

\subsection{Prediction-Focused Pre-training}
\label{pf_pre}
If the learning model is very complicated, for example, a deep neural network, it can be hard to train it directly via the decision-focused learning pipeline. Therefore, to facilitate convergence, we first pre-train the learning model in a prediction-focused fashion.

In \cref{def_soc}, function $g_i(\boldsymbol{y},\boldsymbol{\alpha}_i)$ captures distributional features of the uncertainty $\boldsymbol{y}$. Therefore, lower $\mathbb{E}_{\boldsymbol{y}\sim \mathbb{Q}}[g_i(\boldsymbol{y},\boldsymbol{\alpha}_i)]$ can be deemed as better characterization of these distributional features, \emph{i.e.}, better prediction.

By the above observation, we define the loss function of prediction-focused pre-training as
\begin{equation}\nonumber
    \text{Loss}=\sum_{i=1}^I\Big\|   \mathbb{E}_{ \mathbb{Q}_{\boldsymbol{z}}}[g_i(\boldsymbol{y},\boldsymbol{\alpha}_i)]     \Big\|+\Big\|   \mathbb{E}_{ \mathbb{Q}_{\boldsymbol{z}}}[g_i(\boldsymbol{y},\boldsymbol{\alpha}_i)]-\sigma_i     \Big\|
\end{equation}
where $\mathbb{Q}_{\boldsymbol{z}}$ is defined in \cref{W_SAA}.

\section{Experiments}
\begin{figure*}[tp]
\begin{center}
\centerline{\includegraphics[width=0.95\textwidth]{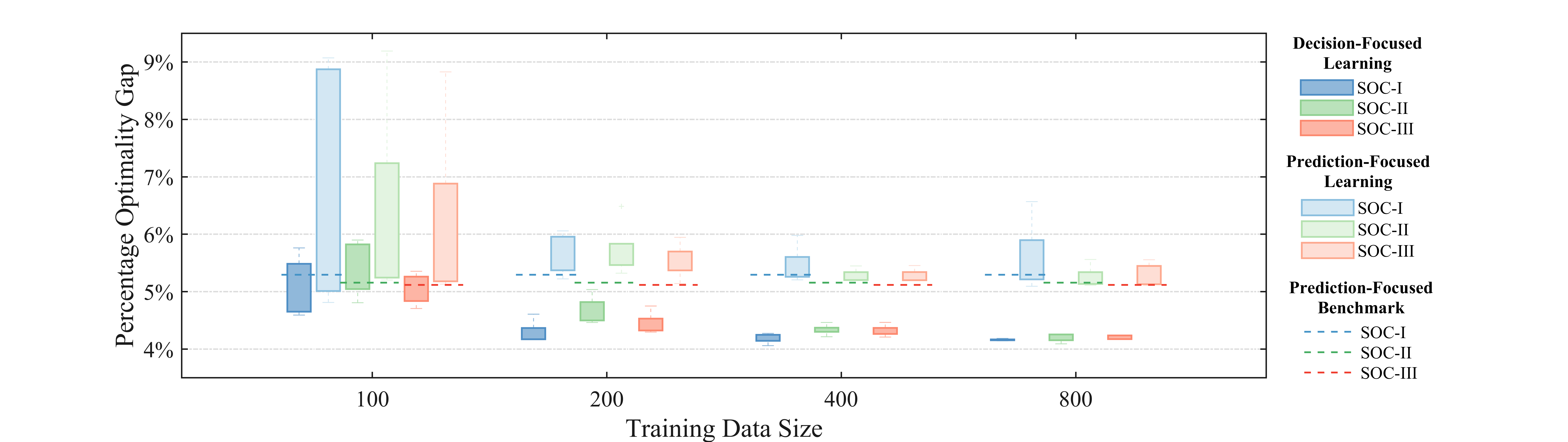}}
\vspace{-5pt}
\caption{Experimental results on the multi-item newsvendor problem.}
\label{news_res}
\end{center}
\vspace{-20pt}
\end{figure*}
To validate the effectiveness of the proposed differentiable DRO layers, we conduct experiments on a toy example and the portfolio management problem\footnote{Source code of all the experiments is available at \href{https://github.com/DOCU-Lab/Differentiable_DRO_Layers}{https://github.com/DOCU-Lab/Differentiable\_DRO\_Layers}.}.

In both experiments, problems are formulated in a contextual DRO setting and thus we can apply the decision-focused learning pipeline developed in \cref{C_DRO}. The detailed experiment setup is presented in \cref{exp_set}.

\subsection{Toy Example: Multi-item Newsvendor Problem}

We consider a multi-item newsvendor problem (\ref{news_pro}) where two options are provided for buying each item, \emph{i.e.}, retail and wholesale. The wholesale price $\boldsymbol{a}^d_i$ is lower than the retail price $\boldsymbol{a}^c_i$ but it can only be sold at a fixed amount $\boldsymbol{v}_i$.\allowdisplaybreaks
    \begin{align}
        \min_{\boldsymbol{x}^c,\boldsymbol{x}^d}&\bigg\{
            \sum_{i=1}^n a_i^c x_i^c + a_i^d v_i x_i^d +b_i(y_i - x_i^c-v_ix_i^d)^{+}\nonumber\\&\hspace{40pt}+d_i(x_i^c+v_ix_i^d-y_i)^+ 
         \bigg\}\nonumber\\
        \text{s.t. }&\hspace{25pt} x_i^c\geq 0, \ x_i^d\in \{ 0,1  \}, \ \forall i \in [n]\label{news_pro}
    \end{align}
where $y_i$ is the demand of item $i$, the continuous variable $\boldsymbol{x}^c$ denotes amount of item bought from retail, integer variable $\boldsymbol{x}^d$ denotes the wholesale option, $\boldsymbol{b}$ is the unit price of additional ordering, and $\boldsymbol{d}$ is the unit holding cost. 

To characterize the uncertainty in demand $y_i$, we consider the following three types of parameterized SOC ambiguity sets, where first- and second-order moment features are characterized.\allowdisplaybreaks
\begin{align}
\allowdisplaybreaks
    \tag{SOC-I}\label{SOC_I}
    \mathscr{U}_{\text{I}}&(\boldsymbol{\mu}_\text{I},\sigma_\text{I})=\left\{ \mathbb{P} \left|\ \begin{gathered}
        \mathbb{P}(\varXi)=1\\
            \mathbb{E}_{\mathbb{P}} \left[\left\| \boldsymbol{y}-\boldsymbol{\mu}_\text{I}  \right\|_1\right]\leq \sigma_\text{I}
        \end{gathered}
         \right.       \right\}\\
    \tag{SOC-II}\label{SOC_II}
    \mathscr{U}_{\text{II}}&(\boldsymbol{\mu}_\text{II},\sigma_\text{II})=\left\{ \mathbb{P} \left|\ \begin{gathered}
        \mathbb{P}(\varXi)=1\\
            \mathbb{E}_{\mathbb{P}} \left[\left\| \boldsymbol{y}-\boldsymbol{\mu}_\text{II}  \right\|_2^2 \right]\leq \sigma_\text{II}
        \end{gathered}
         \right.       \right\}\\
    \tag{SOC-III}\label{SOC_III}
    \mathscr{U}_{\text{III}}&(\boldsymbol{\mu}_\text{I},\sigma_\text{I},\boldsymbol{\mu}_\text{II},\sigma_\text{II})
    =
    \mathscr{U}_{\text{I}}(\boldsymbol{\mu}_\text{I},\sigma_\text{I})\bigcap \mathscr{U}_{\text{II}}(\boldsymbol{\mu}_\text{II},\sigma_\text{II})
\end{align}
The ambiguity set parameters are learned by NN in all the experiments.
We compare the proposed decision-focused learning method (\cref{df_pipe}) with the prediction-focused learning method, which is described in \cref{pf_pre}.

To fully verify the superior performance of the decision-focused learning method, we further compare it with the prediction-focused benchmark. The parameters of the ambiguity set in the prediction-focused benchmark are selected with the full knowledge of the conditional distribution $p(\boldsymbol{y}|\boldsymbol{z})$. For example, in SOC-I, the parameters $\boldsymbol{\mu}_{\text{I}}$ and $\sigma_{\text{I}}$ are directly set to the mean and first-order absolute central moment of the conditional distribution $p(\boldsymbol{y}|\boldsymbol{z})$. Therefore, the performance of the prediction-focused benchmark is the `optimal' performance that the prediction-focused learning method can expect.

In experiments, we take $n=4$ and conduct 10 runs for cases of different training data sizes. We use the percentage optimality gap, whose definition can be found in \cref{exp_set}, to measure the performance of each method. 
 The experiment results are presented in \cref{news_res}.

By \cref{news_res}, the performance of prediction-focused learning will converge to the prediction-focused benchmark as the training data size grows. On the contrary, by directly minimizing the decision loss, the proposed decision-focused learning method demonstrates better performance. Quantitatively, the proposed decision-focused learning method demonstrates average improvements of 21.4\%, 18.7\%, and 18.1\% compared with the prediction-focused benchmark in the three ambiguity sets, respectively.
\subsection{Portfolio Management Problem}

As mentioned in the literature review, \citet{costa2023distributionally} also developed an end-to-end DRO method for portfolio management problems, but their method only applies to pure continuous decision cases. Therefore, in this part, we make a direct comparison with the method proposed in \citet{costa2023distributionally} on the portfolio management problem with pure continuous decisions, and then we further conduct experiments with mixed-integer decisions.

\subsubsection{Portfolio Management Problem with Pure Continuous Decisions}
The portfolio management problem (\ref{port_prob}) with pure continuous decisions aims to select the optimal portfolio $\boldsymbol{x}\in \mathbb{R}^n$ that minimizes the cost and the uncertainty comes from the return $\boldsymbol{y}$. \allowdisplaybreaks
    \begin{gather}
    \begin{gathered}
        \min_{\boldsymbol{x}}\hspace{10pt} -\boldsymbol{y}^T\boldsymbol{x} \\
        \text{s.t.}\hspace{10pt} \boldsymbol{1}^T\boldsymbol{x}=1,\ \boldsymbol{x}\geq 0
    \end{gathered}
        \label{port_prob}
    \end{gather}

    To characterize the uncertainty in return $\boldsymbol{y}$, we consider
the following parameterized SOC ambiguity set.
\begin{equation}\label{amb_mixed_prot}
    \mathscr{U}(\boldsymbol{\mu},\boldsymbol{\sigma})=\left\{ \mathbb{P} \left|\ \begin{gathered}
        \mathbb{P}(\varXi)=1\\
            \mathbb{E}_{\mathbb{P}} \left[\left\| \boldsymbol{y}_i-\boldsymbol{\mu}_i  \right\|_2^2 \right]\leq \boldsymbol{\sigma}_i,\forall  i \in [n]
        \end{gathered}
         \right.       \right\}
\end{equation}

In experiments, we take the dimension $n$ of assets to $40$. For learning the ambiguity set parameter, we use the same 2-layer neural network in both our method and the method proposed in \citet{costa2023distributionally} for a fair comparison.

We use 2,500 data for training, 500 data for validation, and 1,000 data for testing. The average percentage profit of the proposed decision-focused learning method, prediction-focused learning method, and the method proposed in \citet{costa2023distributionally} are 0.289, 0.082, and 0.268, respectively. We also plot the wealth evolution in \cref{costa_compare}. This superiority in performance is ascribed to the fact that a restrictive assumption on the structure of uncertainty distribution is presumed in the method of \citet{costa2023distributionally} while our method applies to general distributions.

\begin{figure}[btph!]
\begin{center}
\centerline{\includegraphics[width=0.45\textwidth]{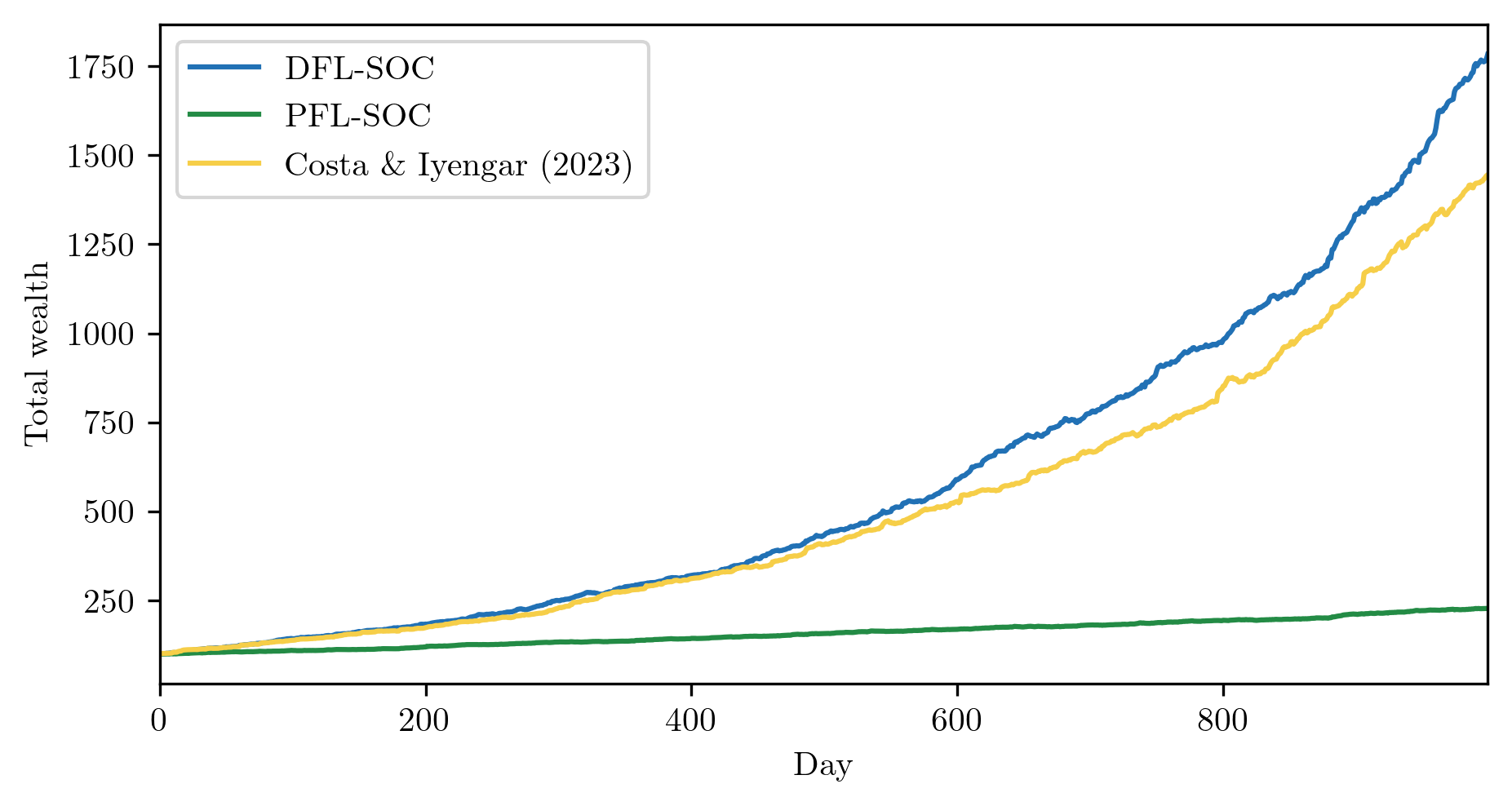}}
\vspace{-10pt}
\caption{Wealth evolution on a 40-dimensional continuous portfolio management problem using decision-focused learning, prediction-focused learning, and method proposed in \citet{costa2023distributionally}. (In the legend, `DFL' stands for decision-focused learning, and `PFL' stands for prediction-focused learning.)}
\label{costa_compare}
\end{center}
\vspace{-25pt}
\end{figure}

\subsubsection{Portfolio Management Problem with Mixed-Integer Decisions}
We further conduct experiments on the mixed-integer portfolio management problem (\ref{port_prob_mix}), where some of the assets are only allowed to
be bought with either a fixed amount or 0. Thus, the decisions on these assets are binary variables. 
\begin{gather}
\begin{gathered}
    \min_{\boldsymbol{x}_c,\boldsymbol{x}_d}\hspace{10pt} -\boldsymbol{y}_c^T\boldsymbol{x}_c - \boldsymbol{y}_d^T\text{diag}(\boldsymbol{v})\boldsymbol{x}_d\\
        \text{s.t.}\hspace{10pt} \boldsymbol{1}^T\boldsymbol{x}_c + \boldsymbol{v}^T\boldsymbol{x}_d=1,\ \boldsymbol{x}_c\geq 0,\boldsymbol{x}_d \in \{ 0,1\}
\end{gathered}
    \label{port_prob_mix}
    \end{gather}
where $\boldsymbol{y}_c,\boldsymbol{y}_d$ are returns corresponding to assets with continuous decisions $\boldsymbol{x}_c$ and binary decisions $\boldsymbol{x}_d$, and $\boldsymbol{v}$ denotes the fixed amount of assets allowed to be bought.  
    
In experiments, the SOC ambiguity set is set to (\ref{amb_mixed_prot}) if not explicitly specified and the number of binary decisions is set to $\frac{1}{5}$ of the number of the problem dimension.

\textbf{Performance with Different Dimensions: }We compare the performance of the proposed decision-focused learning method with prediction-focused learning method on mixed-integer portfolio management problems of different dimensions, and the results are presented in \cref{tab_dimen}, where problem dimension refers to the number of assets. Generally, the performance gap between decision-focused and predict-focused learning methods scales as the problem dimension becomes larger. The advantage of decision-focused method becomes more apparent as the problem dimension increases.

\begin{table}[htpb!]
\vspace{-6pt}
\centering
\caption{Average percentage profit of decision-focused learning and prediction-focused learning on mixed-integer portfolio management problems with different dimensions.}
\begin{tabular}{cccc} 
\toprule
\multirow{2}*{Problem dimension} & \multicolumn{2}{c}{Method}  &\multirow{2}*{Improvement}  \\
\cmidrule(lr){2-3}  
   &    DFL & PFL &  \\
  \midrule
  20 & 0.1561  & 0.0583 & 167\%\\
  40 & 0.1763 & 0.0634 & 178\%\\
  60 & 0.2145 & 0.0479 & 347\%\\
  \bottomrule
  \end{tabular}
  \label{tab_dimen}
\end{table}

\textbf{Performance with Different SOC Ambiguity Sets: }The the SOC ambiguity set (\ref{DEF_SOC_AMB}) is determined by the constraints $g_i$, which significantly affect the performance. Therefore, we conduct experiments on 60-dimensional mixed-integer portfolio management problems with three SOC ambiguity sets with different numbers of constraints. The detailed information of these three ambiguity sets is presented in \cref{exp_set}.

The performance of decision-focused and prediction-focused learning methods using these three types of ambiguity sets are presented in \cref{tab_amb}. With more complex ambiguity sets, both decision-focused and prediction-focused methods have better performance, and the improvement for decision-focused learning generally grows.

\begin{table}[htpb!]
\vspace{-10pt}
\centering
\caption{Average percentage profit of decision-focused learning and prediction-focused learning on 60-dimensional mixed-integer portfolio management problem with different SOC ambiguity sets.}
\begin{tabular}{cccc} 
\toprule
\multirow{2}*{SOC constraint No.} & \multicolumn{2}{c}{Method}  &\multirow{2}*{Improvement}  \\
\cmidrule(lr){2-3}  
   &    DFL & PFL &  \\
  \midrule
  15 & 0.0762  & 0.0357 & 113\% \\
  30 & 0.1349 & 0.0491 & 174\% \\
  60 & 0.2145 & 0.0479 & 347\% \\
  \bottomrule
  \end{tabular}
  \label{tab_amb}
\end{table}

\section{Discussion of Limitations}
The major limitation of applying the proposed DRO-Layer method to large-scale problems lies in the heavy computational burden.
Specifically, 
the decision-focused learning pipeline we developed in Section 4 can be decomposed into four processes: 1. learning layer; 2. projection layer; 3. solving MICP; 4. DRO Layer. Processes 2 and 4 are built on Cvxpylayers \cite{agrawal2019differentiable}, and Process 3 is built on commercial solver Gurobi. To test the computational efficiency and scalability, we conduct experiments and present the computational time of each of these four processes in \cref{diss_limm}.

It is noteworthy that both Cvxpylayers and Gurobi are built on CPU rather than on GPU, thereby leading to computational inefficiency and relatively weak scalability inevitably. If GPU training is allowable in these software, we believe the computation will not be a burden.

\section{Conclusion}
We developed the first generic differentiable DRO layers, where a novel dual-view methodology was proposed to handle the mixed-integer decision via distinct principles. Based on the proposed differentiable DRO layers, we further developed a decision-focused learning pipeline for contextual DRO problems and verified its effectiveness in experiments.

\newpage
\section*{Acknowledgements}
This work was supported in part by the National Natural Science Foundation of China under Grant 62103264, and in part by the National Natural Science Foundation of China (Basic Science Center Program) under Grant 61988101.

\section*{Impact Statement}
This paper presents work whose goal is to advance the integration of DRO and machine learning. None of the potential impacts of our work we feel must be specifically highlighted here.

\bibliography{main}
\bibliographystyle{icml2024}

\newpage
\appendix
\onecolumn

\section{Supplementary Background Material}

In this section, we give supplementary information on DRO with SOC ambiguity set in \ref{SOC_R_S}, \ref{exam_SOC_amb}, and \ref{Slater}. Most of these materials are selected from \citet{ben2001lectures} and \citet{bertsimas2019adaptive}.

In \cref{app_diss_unique}, \ref{app_imp_samp}, \ref{inte_cut}, and \ref{ana_pro}, detailed information on our method are provided.

\subsection{Second-Order Cone Representable Set}
\label{SOC_R_S}
[\citet{ben2001lectures}, p.86] A set $\mathscr{W}\subset \mathbb{R}^n$ is a second-order cone representable set if there exists $J$ second-order cone inequalities of the form
\begin{equation}
    \boldsymbol{A}_j\left[ \begin{gathered}
        \boldsymbol{y}\\ \boldsymbol{v}
    \end{gathered} \right]-\boldsymbol{b}_j \geq_{L^{m_j}} \boldsymbol{0},\forall j \in [J]
\end{equation}
such that 
\begin{equation}
    \boldsymbol{y}\in \mathscr{W}\iff \exists \boldsymbol{v}:  \boldsymbol{A}_j\left[ \begin{gathered}
        \boldsymbol{y}\\ \boldsymbol{v}
    \end{gathered} \right]-\boldsymbol{b}_j \geq_{L^{m_j}} \boldsymbol{0},\forall j \in [J]
\end{equation}
where $L^{m}$ is the $m$ dimensional second-order cone:
\begin{equation}
    L^{m}=\Big\{ \boldsymbol{y}=(y_1,\cdots,y_m)^T\in \mathbb{R}^m \Big| y_m\geq \sqrt{y_1^2+\cdots + y_{m-1}^2}  \Big\}
\end{equation}

\subsection{Examples of Parameterized SOC Ambiguity Set}
\label{exam_SOC_amb}
The parameterized SOC ambiguity set is determined by its constraints $\mathbb{E}_{\mathbb{P}}[g_i(\boldsymbol{y},\boldsymbol{\alpha}_i)]\leq \sigma_i$, and here we present some choices of function $g$ and the parameter $\boldsymbol{\alpha}$.
\begin{enumerate}
    \item $g=\boldsymbol{\mu}^T\boldsymbol{y}$ with vector $\boldsymbol{\mu}$ as the parameter.
    \item $g=|\boldsymbol{\mu}^T\boldsymbol{y}-h|$ with vector $\boldsymbol{\mu}$ and scalar $h$ as the parameter.
    \item $g=(|\boldsymbol{\mu}^T\boldsymbol{y}-h|)^p$ for some rational $p\geq 1$ with vector $\boldsymbol{\mu}$ and scalar $h$ as the parameter.
    \item $g=((\boldsymbol{\mu}^T\boldsymbol{y}-h)^+)^2=(\max \{  0, \boldsymbol{\mu}^T\boldsymbol{y}-h\})^2$ vector $\boldsymbol{\mu}$ and scalar $h$ as the parameter.
    \item $g=\left\| \boldsymbol{A}\boldsymbol{y}-\boldsymbol{\mu}  \right\|_p$ for some rational $p\geq 1$ norm $\| \cdot\|_p$ with matrix $\boldsymbol{A}$ and vector $\boldsymbol{\mu}$ as the parameter.
\end{enumerate}

More examples can be constructed by taking the maximum, \emph{i.e.}, $g=\max_{l\in [L]}g_i$, and non-negtive sum, \emph{i.e.}, $g=\sum_{l=1}^L\lambda_ig_i$ for $\lambda_i\geq 0$. See \citet{ben2001lectures} for more operators that preserve the second-order cone representable property of $g$. 

\subsection{Slater’s Condition for Parameterized SOC Ambiguity Set}
\label{Slater}
According to Proposition 1 in \citet{bertsimas2019adaptive}, the parameterized SOC ambiguity set $\mathscr{U}(\boldsymbol{\theta})$ (defined in \cref{def_soc}) can be equivalently reformulated as follows.
\begin{equation}
    \mathscr{U}(\boldsymbol{\theta})=\left\{ \#_{\boldsymbol{y}}\mathbb{Q} \left|\  \begin{gathered}
        (\boldsymbol{y},\boldsymbol{u})\sim \mathbb{Q} \\
        \mathbb{E}_{\mathbb{Q}} [u_i]\leq \sigma_i, \forall i \in [I]\\
        \mathbb{Q}(\mathscr{V})=1
    \end{gathered}  \right.         \right\}
\end{equation}
where each dimension $u_i$ of variable $\boldsymbol{u}$ corresponds to the constraint $g_i$ in  $\mathscr{U}(\boldsymbol{\theta})$, distribution $\mathbb{Q}$ is on the space of $(\boldsymbol{y},\boldsymbol{u})$, $\#_{\boldsymbol{y}}\mathbb{Q}$ is the marginal distribution of $\mathbb{Q}$ on dimension $\boldsymbol{y}$, and the support $\mathscr{V}$ is defined as
\begin{equation}\label{app_sla}
    \mathscr{V}=\{ (\boldsymbol{y},\boldsymbol{u}) | \boldsymbol{y}\in \varXi, g_i(\boldsymbol{y},\boldsymbol{\alpha}_i)\leq u_i,\forall i \in [I]  \}=\{(\boldsymbol{y},\boldsymbol{u}) | \boldsymbol{y}\in \varXi \}\bigcap_{i=1}^I \{ (\boldsymbol{y},\boldsymbol{u}) | g_i(\boldsymbol{y},\boldsymbol{\alpha}_i)\leq u_i \}
\end{equation}

By \cref{def_soc}, $\{(\boldsymbol{y},\boldsymbol{u}) | \boldsymbol{y}\in \varXi \}$ is second-order cone representable set, and each $\{ (\boldsymbol{y},\boldsymbol{u}) | g_i(\boldsymbol{y},\boldsymbol{\alpha}_i)\leq u_i \}$ are differentiable parameterized second-order representable sets. Therefore, $\mathscr{V}$ is also a differentiable parameterized second-order representable set, so $\mathscr{V}$ can be represented by finitely many second-order cone constraints.
\begin{equation}\label{appp_slaa}
    \mathscr{V}=\left\{(\boldsymbol{y},\boldsymbol{u}) \left|\exists \boldsymbol{v}:  \boldsymbol{A}_j(\boldsymbol{\theta})\left[ \begin{gathered}
        \boldsymbol{y}\\ \boldsymbol{u}\\\boldsymbol{v}
    \end{gathered} \right]-\boldsymbol{b}_j(\boldsymbol{\theta}) \geq_{L^{m_j}} \boldsymbol{0},\forall j \in [J]  \right. \right\}
\end{equation}

The Slater’s condition requires that there exist a $(\boldsymbol{y}^{\ast},\boldsymbol{u}^{\ast},\boldsymbol{v}^{\ast})$ such that $u_i^{\ast}< \sigma_i,\ \forall i \in [I]$ and
\begin{equation}
    \boldsymbol{A}_j(\boldsymbol{\theta})\left[ \begin{gathered}
        \boldsymbol{y}^{\ast}\\ \boldsymbol{u}^{\ast}\\\boldsymbol{v}^{\ast}
    \end{gathered} \right]-\boldsymbol{b}_j(\boldsymbol{\theta}) >_{L^{m_j}} \boldsymbol{0},\forall j \in [J] 
\end{equation}

\subsection{Discussion of \cref{assum_unique_sol}}
\label{app_diss_unique}
Assumption (i) ensures the uniqueness of the continuous decision, which is common in differentiable optimization layer research \cite{agrawal2019differentiable}. 

For the discrete decision $\boldsymbol{x}_d$, since its feasible region $\mathcal{X}_d$ is finite, we can not require the uniqueness of $\boldsymbol{x}_d$ for all $\boldsymbol{\theta}\in \Theta$, where $\boldsymbol{\theta}=M_{\boldsymbol{\phi}}(\boldsymbol{z})$. For example, if we consider the combinatorial optimization problem, \emph{i.e.},
\begin{equation}\label{app_comb}
    \boldsymbol{x}_d=\argmin_{\boldsymbol{x}_d\in \{ 0,1 \}^n } f(\boldsymbol{x}_d,\boldsymbol{\theta}):= \boldsymbol{\theta}^T \boldsymbol{x}_d, \text{ s.t. } \boldsymbol{1}^T\boldsymbol{x}_d=1
\end{equation}
The solution to this problem is not unique when the prediction $\boldsymbol{\theta}$ has multiple minimum elements.

However, we note that for problem (\ref{app_comb}), the $\boldsymbol{\theta}$ leading to multiple solutions has measure zero in its space $\Theta$, and this property holds for a lot of problems. Therefore, we require in assumption (ii) the uniqueness to hold almost surely with respect to the marginal distribution $\mathbb{P}_{\boldsymbol{z}}$ of $\boldsymbol{z}$.

Specifically, in combinatorial optimization problem (\ref{app_comb}), if the learning model is a linear one, \emph{i.e.}, $\boldsymbol{\theta}=\boldsymbol{A} \boldsymbol{z}+\boldsymbol{b}$ with $\boldsymbol{\phi}=(\boldsymbol{A},\boldsymbol{b})$, where the rows of matrix $\boldsymbol{A}$ are all different, and covariate $\boldsymbol{z}$ has marginal distribution absolutely continuous with respect to the Lebesgue measure, then (ii) is satisfied.

We further note that (ii) is also implicitly assumed in \citet{Pogan}. 

\subsection{Self-Normalized  Importance Sampling}
\label{app_imp_samp}
Suppose we want to compute the expectation of a random variable $J(\boldsymbol{x}_d)$, \emph{i.e.},
\begin{equation}
    \mathbb{E}_{\boldsymbol{x_d}\sim p(\boldsymbol{x_d}|\boldsymbol{\theta},\lambda)} [J(\boldsymbol{x_d})]
\end{equation}
where $p(\boldsymbol{x_d}|\boldsymbol{\theta},\lambda)$ is defined as in \cref{energy_distri}.

The importance sampling aims to compute $\mathbb{E}_{\boldsymbol{x_d}\sim p(\boldsymbol{x_d}|\boldsymbol{\theta},\lambda)} [J(\boldsymbol{x_d})]$ by leveraging a proposal distribution $q$ which is absolutely continuous with respect to $p(\boldsymbol{x_d}|\boldsymbol{\theta},\lambda)$. With proposal distribution $q$, the above expectation can be computed by
\begin{equation}\nonumber
    \mathbb{E}_{\boldsymbol{x_d}\sim p(\boldsymbol{x_d}|\boldsymbol{\theta},\lambda)} [J(\boldsymbol{x_d})]=\frac{\mathbb{E}_{\boldsymbol{x_d}\sim q(\boldsymbol{x_d})}\bigg[ \frac{E(\boldsymbol{x_d},\boldsymbol{\theta},\lambda)}{q(\boldsymbol{x_d})}J(\boldsymbol{x_d}) \bigg]}{\mathbb{E}_{\boldsymbol{x_d}\sim q(\boldsymbol{x_d})}\bigg[ \frac{E(\boldsymbol{x_d},\boldsymbol{\theta},\lambda)}{q(\boldsymbol{x_d})} \bigg] }
\end{equation}

Therefore, by sampling from the known distribution $q$, the expectation can be estimated unbiasedly.

\subsection{Oracles to Get $T$ Optimal Integer Solutions}
\label{inte_cut}
Suppose that we want to solve $T$ optimal integer solutions of the following mixed-integer linear cone program. 
\begin{equation}\label{pro_pro_exa}
\begin{gathered}
    \min_{\boldsymbol{x}_d,\boldsymbol{x_c},\boldsymbol{v}} \boldsymbol{a}^T\boldsymbol{x}_d+\boldsymbol{b}^T\boldsymbol{x}_c+\boldsymbol{c}^T\boldsymbol{v}\\
    \text{s.t. }  \boldsymbol{x}_d\in \{ 0,1 \}^{n_d}, \ \boldsymbol{x}_c \in \mathbb{R}^{n_c}
    \\\boldsymbol{A}_i^T\boldsymbol{x}_d+\boldsymbol{B}_i^T\boldsymbol{x}_c+\boldsymbol{C}_i^T\boldsymbol{v}\leq_{\mathscr{K}_i} \boldsymbol{0},\forall i \in [I]
\end{gathered}    
\end{equation}

Such mixed-integer linear cone program can be solved by commercial solvers like \emph{gurobi}. We solve this program and get the optimal integer solution $\boldsymbol{x}_d^1$, and then we add the following constraint to problem (\ref{pro_pro_exa}).
\begin{equation}\label{cut_cons}
    \Big(\boldsymbol{1}-2\boldsymbol{x}_d^1 \Big)^T\boldsymbol{x}_d+\boldsymbol{1}^T\boldsymbol{x}_d^1\geq 1
\end{equation}

Constraint (\ref{cut_cons}) only cuts out $\boldsymbol{x}_d^1$ from the feasible region. Therefore, by solving (\ref{pro_pro_exa}) with extra constraint (\ref{cut_cons}) we will get the second optimal solution $\boldsymbol{x}_d^2$. Then we cut out $\boldsymbol{x}_d^2$ to get $\boldsymbol{x}_d^3$.

Repeat the above process for $T$ times and we will get $T$ optimal integer solutions.

\subsection{Analysis of the Projection Layer}
\label{ana_pro}
If $g_i,i\in [I]$ in the parameterized SOC ambiguity set are selected as what we give in \cref{exam_SOC_amb}, then the epigraph of $g_i(\boldsymbol{y},\boldsymbol{\alpha}_i)$ with respect to $\boldsymbol{\alpha}_i$, \emph{i.e.},
\begin{equation}
    \{ (\boldsymbol{\alpha}_i,u)|\boldsymbol{\alpha}_i\in \mathcal{A}_i,u\geq    g_i(\boldsymbol{y},\boldsymbol{\alpha}_i)       \}
\end{equation}
is also a second-order cone representable set.

By \citet{ben2001lectures} p.91, this representability is preserved by a non-negative sum, so
\begin{equation}
    \sum_{n=1}^N \omega_{n}(\boldsymbol{z}) g_i(\boldsymbol{y}_n,\boldsymbol{\alpha}_i)\leq \sigma_i
\end{equation}
can be expressed by finitely many second-order cone constraints.

\section{Proofs}
\subsection{Proof of \cref{thm_refor}}
\label{proof_refor}
When the cost function $c(\boldsymbol{x},\boldsymbol{y})$ is of the form (ii) in \cref{ass_cost_func}, then \cref{thm_refor} coincides with Theorem 1 in \citet{bertsimas2019adaptive}.

For the case the cost function $c(\boldsymbol{x},\boldsymbol{y})$ is of the form (i) in \cref{ass_cost_func}, the proof is quite similar. Since in this case $c(\boldsymbol{x},\boldsymbol{y})$ is linear in $\boldsymbol{y}$, according to Theorem 1 in \citet{bertsimas2019adaptive}, the worst-case expectation $f(\boldsymbol{x},\mathscr{U}(\boldsymbol{\theta}))=\max_{\mathbb{P}\in \mathscr{U}(\boldsymbol{\theta})} \mathbb{E}_{\boldsymbol{y}\sim \mathbb{P}}[c(\boldsymbol{x},\boldsymbol{y})]$ is equivalent to
\begin{gather}
    f(\boldsymbol{x},\mathscr{U}(\boldsymbol{\theta}))=\min_{r,\boldsymbol{\beta}} {r+\boldsymbol{\beta}^T \boldsymbol{\sigma}}\\
    \text{s.t. } r\geq c(\boldsymbol{x},\boldsymbol{y}) - \boldsymbol{\beta}^T \boldsymbol{u},\ \forall (\boldsymbol{y},\boldsymbol{u})\in \mathscr{V},\boldsymbol{y}\geq \boldsymbol{0}\label{prrof_1_conss}\\
    \boldsymbol{\beta}\geq \boldsymbol{0}
\end{gather}
where $\mathscr{V}$ is defined in \cref{app_sla}.

Since \cref{ass_sla} holds, we can leverage the duality theory to reformulate constraint (\ref{prrof_1_conss}) as follows.
\begin{gather}
    r\geq \max_{\boldsymbol{y},\boldsymbol{u}\in \mathscr{V},\boldsymbol{y}\geq \boldsymbol{0}} c(\boldsymbol{x},\boldsymbol{y}) - \boldsymbol{\beta}^T \boldsymbol{u}\iff r\geq \max_{\boldsymbol{y},\boldsymbol{u}\in \mathscr{V},\boldsymbol{y}\geq \boldsymbol{0}} 
 \sum_{k=1}^{K} c_k(\boldsymbol{x})y_k- \boldsymbol{\beta}^T \boldsymbol{u} 
 \\
 \iff r\geq \max_{\boldsymbol{y}\geq \boldsymbol{0}} \min_{\boldsymbol{\eta}_j\geq_{L^{m_j}} \boldsymbol{0}}\left\{  \sum_{k=1}^{K} c_k(\boldsymbol{x})y_k- \boldsymbol{\beta}^T \boldsymbol{u} +\sum_{j=1}^{J}\boldsymbol{\eta}_j^T\left( \boldsymbol{A}_j^{\boldsymbol{y}}(\boldsymbol{\theta})\boldsymbol{y}+\boldsymbol{A}_j^{\boldsymbol{u}}(\boldsymbol{\theta})\boldsymbol{u}+\boldsymbol{A}_j^{\boldsymbol{v}}(\boldsymbol{\theta})\boldsymbol{v}-\boldsymbol{b}_j(\boldsymbol{\theta})  \right) \right\}
 \\
 \iff r\geq \min_{\boldsymbol{\eta}_j\geq_{L^{m_j}} \boldsymbol{0}} \max_{\boldsymbol{y}\geq \boldsymbol{0}}\left\{  \sum_{k=1}^{K} c_k(\boldsymbol{x})y_k- \boldsymbol{\beta}^T \boldsymbol{u} +\sum_{j=1}^{J}\boldsymbol{\eta}_j^T\left( \boldsymbol{A}_j^{\boldsymbol{y}}(\boldsymbol{\theta})\boldsymbol{y}+\boldsymbol{A}_j^{\boldsymbol{u}}(\boldsymbol{\theta})\boldsymbol{u}+\boldsymbol{A}_j^{\boldsymbol{v}}(\boldsymbol{\theta})\boldsymbol{v}-\boldsymbol{b}_j(\boldsymbol{\theta})  \right) \right\}\label{proof_1_dual}
 \\
 \iff
     r\geq -\sum_{j=1}^{J}\boldsymbol{\eta}_j^T\boldsymbol{b}_j(\boldsymbol{\theta}),
      -\sum_{j=1}^{J} (\boldsymbol{A}_j^{\boldsymbol{y}}(\boldsymbol{\theta}))^T \boldsymbol{\eta}_j \geq \left[ \begin{gathered}
         c_1(\boldsymbol{x})\\
         \vdots\\
         c_K(\boldsymbol{x})
     \end{gathered} \right],  \sum_{j=1}^{J} (\boldsymbol{A}_j^{\boldsymbol{u}}(\boldsymbol{\theta}))^T \boldsymbol{\eta}_j=\boldsymbol{\beta}, \sum_{j=1}^{J} (\boldsymbol{A}_j^{\boldsymbol{v}}(\boldsymbol{\theta}))^T \boldsymbol{\eta}_j=\boldsymbol{0},\boldsymbol{\eta}_j\geq_{L^{m_j}}\boldsymbol{0}\label{proof_1_1}
\end{gather}
where $(\boldsymbol{A}_j^{\boldsymbol{y}}(\boldsymbol{\theta}),\boldsymbol{A}_j^{\boldsymbol{u}}(\boldsymbol{\theta}),\boldsymbol{A}_j^{\boldsymbol{v}}(\boldsymbol{\theta}))=\boldsymbol{A}_j(\boldsymbol{\theta})$ and $\boldsymbol{b}_j (\boldsymbol{\theta})$ are defined in \cref{appp_slaa}, and (\ref{proof_1_dual}) hold by duality theory since the Slater's condition holds by \cref{ass_sla}.

The second term in (\ref{proof_1_1}) can be reformulated as
\begin{gather}
      -\sum_{j=1}^{J} (\boldsymbol{A}_j^{\boldsymbol{y}}(\boldsymbol{\theta})\boldsymbol{e}_k)^T \boldsymbol{\eta}_j \geq c_k(\boldsymbol{x}),\ \forall k\in [K]\label{proof_1_2}
\end{gather}
where $\boldsymbol{e}_k$ is the vector where the $k$th element is 1 and other elements are 0. Since by \cref{ass_cost_func} the epigraph of $c_i(\boldsymbol{x})$ is a second-order cone representable set, thus each $- \sum_{j=1}^{J} (\boldsymbol{A}_j^{\boldsymbol{y}}(\boldsymbol{\theta})\boldsymbol{e}_k)^T \boldsymbol{\eta}_j \geq c_k(\boldsymbol{x})$ can be expressed by finitely many second-order cone constraints.

Therefore, \cref{thm_refor} holds for both the cases in \cref{ass_cost_func}.

\subsection{Proof of \cref{diff_con}}
\label{proof_diff_con}

By \cref{thm_refor}, the worst-case expectation $f(\boldsymbol{x},\mathscr{U}(\boldsymbol{\theta}))$ is a linear second-order cone program and all the constraints are linear in $\boldsymbol{x}$, so under continuous assumption of $\boldsymbol{x}$, the DRO problem
\begin{equation}
    \min_{\boldsymbol{x}\in \mathcal{X}}f(\boldsymbol{x},\mathscr{U}(\boldsymbol{\theta}))\label{proo_2_DR0}
\end{equation}
is also a linear second-order cone program.

Specifically, by \cref{proof_refor}, if the cost function $c(\boldsymbol{x},\boldsymbol{y})$ is of the form (i) in \cref{ass_cost_func}, (\ref{proo_2_DR0}) is equivalent to \allowdisplaybreaks
\begin{gather}
    \min_{\boldsymbol{x}\in \mathcal{X}}f(\boldsymbol{x},\mathscr{U}(\boldsymbol{\theta}))=\min_{\boldsymbol{x}\in \mathcal{X},r,\boldsymbol{\beta}} {r+\boldsymbol{\beta}^T \boldsymbol{\sigma}}\hspace{80pt}\label{proof_2_soc}\\
    \text{s.t. } r\geq -\sum_{j=1}^{J}\boldsymbol{\eta}_j^T\boldsymbol{b}_j(\boldsymbol{\theta})\\
     - \sum_{j=1}^{J} (\boldsymbol{A}_j^{\boldsymbol{y}}(\boldsymbol{\theta})\boldsymbol{e}_k)^T \boldsymbol{\eta}_j \geq c_k(\boldsymbol{x}),\ \forall k\in [K]\label{proo_2_2}
     \\ - \sum_{j=1}^{J} (\boldsymbol{A}_j^{\boldsymbol{u}}(\boldsymbol{\theta}))^T \boldsymbol{\eta}_j=\boldsymbol{\beta}\\ \sum_{j=1}^{J} (\boldsymbol{A}_j^{\boldsymbol{v}}(\boldsymbol{\theta}))^T \boldsymbol{\eta}_j=\boldsymbol{0},\boldsymbol{\eta}_j\geq_{L^{m_j}}\boldsymbol{0}\\
    \boldsymbol{\beta}\geq \boldsymbol{0},\boldsymbol{\eta}_j\geq_{L^{m_j}}\boldsymbol{0},\forall j\in [J]
\end{gather}
where constraint (\ref{proo_2_2}) is defined in \cref{proof_1_2} in \cref{proof_refor} and can be reformulated into finitely many second-order cone constraints that are linear in $\boldsymbol{x}$.

If the cost function $c(\boldsymbol{x},\boldsymbol{y})$ is of the form (ii) in \cref{ass_cost_func}, a similar reformulation can be derived.

The cone programming $\min_{\boldsymbol{x}\in \mathcal{X}}f(\boldsymbol{x},\mathscr{U}(\boldsymbol{\theta}))$ takes $\boldsymbol{A}_j(\boldsymbol{\theta})=(\boldsymbol{A}_j^{\boldsymbol{y}}(\boldsymbol{\theta}),\boldsymbol{A}_j^{\boldsymbol{u}}(\boldsymbol{\theta}),\boldsymbol{A}_j^{\boldsymbol{v}}(\boldsymbol{\theta}))$ and $\boldsymbol{b}_j(\boldsymbol{\theta})$ as its parameter. Then by \citet{Agrawal2019DifferentiatingTA}, the optimal value $\boldsymbol{x}^{\ast}=\argmin_{\boldsymbol{x}\in \mathcal{X}}f(\boldsymbol{x},\mathscr{U}(\boldsymbol{\theta}))$ is differentiable with respect to parameter $\boldsymbol{A}_j(\boldsymbol{\theta})$ and $\boldsymbol{b}_j(\boldsymbol{\theta})$. Further, by \cref{def_soc_set}, $\boldsymbol{A}_j(\boldsymbol{\theta})$ and $\boldsymbol{b}_j(\boldsymbol{\theta})$ are differentiable with respect to $\boldsymbol{\theta}$. 

Therefore, $\boldsymbol{x}^{\ast}=\argmin_{\boldsymbol{x}\in \mathcal{X}}f(\boldsymbol{x},\mathscr{U}(\boldsymbol{\theta}))$ is differentiable with respect to $\boldsymbol{\theta}$.

\subsection{Proof of \cref{thm_epi_converge}}
\label{proof_thm_epi_converge}
(i) It suffices to prove that for any sequence $\lambda_n\searrow 0$, $\boldsymbol{\phi}\in \Phi$, and sequence $\boldsymbol{\phi}_n\to \boldsymbol{\phi}$, the following equality holds
\begin{equation}
    \lim_{n\to \infty} \mathcal{R}_{\lambda_n}(\boldsymbol{\phi}_n)=\mathcal{R}(\boldsymbol{\phi})
\end{equation}

For ease of notation, we define 
\begin{equation}
    f^{\ast}(\boldsymbol{x}_d,M_{\boldsymbol{\phi}}(\boldsymbol{z}))=\min_{\boldsymbol{x}_c\in \mathcal{X}_c(\boldsymbol{x}_d)}f((\boldsymbol{x}_d,\boldsymbol{x}_c),M_{\boldsymbol{\phi}}(\boldsymbol{z}))=f\big((\boldsymbol{x}_d,\boldsymbol{x}_c^{\ast}(\boldsymbol{x}_d,M_{\boldsymbol{\phi}})),M_{\boldsymbol{\phi}}(\boldsymbol{z})\big)
\end{equation}

By \cref{assum_bound}, $f((\boldsymbol{x}_d,\boldsymbol{x}_c),M_{\boldsymbol{\phi}}(\boldsymbol{z}))$ is continuous in $\boldsymbol{x}_c$ and $\boldsymbol{\phi}$, and by \cref{assum_opt_con}, $\boldsymbol{x}_c^{\ast}(\boldsymbol{x}_d,M_{\boldsymbol{\phi}}(\boldsymbol{z}))$ is also continuous in $\boldsymbol{\phi}$, so $f^{\ast}(\boldsymbol{x}_d,M_{\boldsymbol{\phi}}(\boldsymbol{z}))$ is continuous in $\boldsymbol{\phi}$.

If for a pair of $(\boldsymbol{z},\boldsymbol{\phi})$ the optimal integer solution $\boldsymbol{x}_d^{\ast}(M_{\boldsymbol{\phi}}(\boldsymbol{z}))= \argmin_{\boldsymbol{x}_d\in \mathcal{X}_d} f^{\ast}(\boldsymbol{x}_d,M_{\boldsymbol{\phi}}(\boldsymbol{z}))$ is unique, then it holds that
\begin{equation}\label{unique_intepre}
    \min_{\boldsymbol{x}_d\in \mathcal{X}_d/\{ \boldsymbol{x}_d^{\ast}(M_{\boldsymbol{\phi}}(\boldsymbol{z}))\}} f^{\ast}\Big(\boldsymbol{x}_d,M_{\boldsymbol{\phi}}(\boldsymbol{z})\Big)-f^{\ast}\Big(\boldsymbol{x}_d^{\ast},M_{\boldsymbol{\phi}}(\boldsymbol{z})\Big)>0
\end{equation}

Then by the continuity of $f^{\ast}(\boldsymbol{x}_d,M_{\boldsymbol{\phi}}(\boldsymbol{z}))$ in $\boldsymbol{\phi}$, there exists a $\epsilon>0$ such that
\begin{equation}\label{unique_inte2}
\forall \overline{\boldsymbol{\phi}}\in \left\{\overline{\boldsymbol{\phi}} \left| \| \overline{\boldsymbol{\phi}}-\boldsymbol{\phi}  \|\leq \epsilon  \right. \right\},
    \min_{\boldsymbol{x}_d\in \mathcal{X}_d/\{ \boldsymbol{x}_d^{\ast}(M_{\boldsymbol{\phi}}(\boldsymbol{z}))\}} f^{\ast}\Big(\boldsymbol{x}_d,M_{\overline{\boldsymbol{\phi}}}(\boldsymbol{z})\Big)-f^{\ast}\Big(\boldsymbol{x}_d^{\ast}(M_{\boldsymbol{\phi}}(\boldsymbol{z})),M_{\overline{\boldsymbol{\phi}}}(\boldsymbol{z})\Big)\geq \epsilon
\end{equation}

By the above observation, we further define 
\begin{equation}
    Y(\boldsymbol{\phi},N)=
    \left\{ \boldsymbol{z}\left| \min_{\boldsymbol{x}_d\in \mathcal{X}_d/\{ \boldsymbol{x}_d^{\ast}(M_{\boldsymbol{\phi}}(\boldsymbol{z}))\}} f^{\ast}\Big(\boldsymbol{x}_d,M_{\overline{\boldsymbol{\phi}}}(\boldsymbol{z})\Big)-f^{\ast}\Big(\boldsymbol{x}_d^{\ast}(M_{\boldsymbol{\phi}}(\boldsymbol{z})),M_{\overline{\boldsymbol{\phi}}}(\boldsymbol{z})\Big)\geq \frac{1}{N} ,\forall \overline{\boldsymbol{\phi}}\in \left\{\overline{\boldsymbol{\phi}} \left| \| \overline{\boldsymbol{\phi}}-\boldsymbol{\phi}  \|\leq \frac{1}{N}  \right. \right\} \right.  \right\}\label{thm1eq1}
\end{equation}
It is obvious that $Y(\boldsymbol{\phi},1)\subset Y(\boldsymbol{\phi},2)\subset \cdots \subset Y(\boldsymbol{\phi},N)\subset \cdots$.

By the argument concerning (\ref{unique_inte2}), we conclude that if inequality (\ref{unique_intepre}) holds for a pair of $(\boldsymbol{z},\boldsymbol{\phi})$, then $\boldsymbol{z}\in Y(\boldsymbol{\phi},N)$ for sufficiently large $N$.

Since by \cref{assum_unique_sol}, for any $\boldsymbol{\phi}\in \Phi$, inequality (\ref{unique_intepre}) holds almost surely, thus we have
\begin{equation}
    \mathbb{P}\bigg(\bigcup_{N=1}^\infty Y(\boldsymbol{\phi},N)\bigg)=1
\end{equation}

By \cref{assum_bound}, $l$ is bounded, and we denote this bounded by $\Psi$. 

Therefore, for any $\epsilon>0$, there exist a $N_{\epsilon / \Psi}$ such that 
\begin{equation}
    \mathbb{P}(Y(\boldsymbol{\phi},N_{\epsilon / \Psi}))\geq 1-\epsilon / \Psi
\end{equation}
Therefore, when $n\geq N_{\epsilon / \Psi}$, we have\allowdisplaybreaks
\begin{gather}
    |\mathcal{R}_{\lambda_n}(\boldsymbol{\phi}_n)- \mathcal{R}(\boldsymbol{\phi})|\\
    \begin{gathered}
    =\Bigg| \mathbb{E}_{(\boldsymbol{z},\boldsymbol{y})\sim \mathbb{P}} \Bigg[  \sum_{\boldsymbol{x}_d\in \mathcal{X}_d}p(\boldsymbol{x}_d|M_{\boldsymbol{\phi}_n}(\boldsymbol{z}),\lambda_n) l((\boldsymbol{x}_d,\boldsymbol{x}_c^{\ast}(\boldsymbol{x}_d,M_{\boldsymbol{\phi}_n}(\boldsymbol{z}))),\boldsymbol{y})\\-
    l \bigg( \Big(\boldsymbol{x}_d^{\ast}(M_{\boldsymbol{\phi}}(\boldsymbol{z})),\boldsymbol{x}_c^{\ast}(\boldsymbol{x_d}^{\ast}(M_{\boldsymbol{\phi}}(\boldsymbol{z})),M_{\boldsymbol{\phi}}(\boldsymbol{z}))\Big),\boldsymbol{y}\bigg)\Bigg] \Bigg|
    \end{gathered}\\
    \begin{gathered}
        \leq \Bigg|\mathbb{E}_{(\boldsymbol{z},\boldsymbol{y})\sim \mathbb{P}} \Bigg\{\mathbbm{1}\big(\boldsymbol{z}\in Y(\boldsymbol{\phi},N_{\epsilon / \Psi})\big) \Bigg[  \sum_{\boldsymbol{x}_d\in \mathcal{X}_d}p(\boldsymbol{x}_d|M_{\boldsymbol{\phi}_n}(\boldsymbol{z}),\lambda_n) l((\boldsymbol{x}_d,\boldsymbol{x}_c^{\ast}(\boldsymbol{x}_d,M_{\boldsymbol{\phi}_n}(\boldsymbol{z}))),\boldsymbol{y})
        \\-
    l \bigg( \Big(\boldsymbol{x}_d^{\ast}(M_{\boldsymbol{\phi}}(\boldsymbol{z})),\boldsymbol{x}_c^{\ast}(\boldsymbol{x_d}^{\ast}(M_{\boldsymbol{\phi}}(\boldsymbol{z})),M_{\boldsymbol{\phi}}(\boldsymbol{z}))\Big),\boldsymbol{y}\bigg)\Bigg] \Bigg\}\Bigg|
    \\+
    \left|\mathbb{E}_{(\boldsymbol{z},\boldsymbol{y})\sim \mathbb{P}} \left[\mathbbm{1}\big(\boldsymbol{z}\notin Y(\boldsymbol{\phi},N_{\epsilon / \Psi})\big)  \sum_{\boldsymbol{x}_d\in \mathcal{X}_d}p(\boldsymbol{x}_d|M_{\boldsymbol{\phi}_n}(\boldsymbol{z}),\lambda_n) l((\boldsymbol{x}_d,\boldsymbol{x}_c^{\ast}(\boldsymbol{x}_d,M_{\boldsymbol{\phi}_n}(\boldsymbol{z}))),\boldsymbol{y})\right]\right|
    \\
    + \left|\mathbb{E}_{(\boldsymbol{z},\boldsymbol{y})\sim \mathbb{P}} \left[\mathbbm{1}\big(\boldsymbol{z}\notin Y(\boldsymbol{\phi},N_{\epsilon / \Psi})\big)  l \bigg( \Big(\boldsymbol{x}_d^{\ast}(M_{\boldsymbol{\phi}}(\boldsymbol{z})),\boldsymbol{x}_c^{\ast}(\boldsymbol{x_d}^{\ast}(M_{\boldsymbol{\phi}}(\boldsymbol{z})),M_{\boldsymbol{\phi}}(\boldsymbol{z}))\Big),\boldsymbol{y}\bigg)
    \right]\right| 
    \end{gathered}\\
    \begin{gathered}
         \leq \Bigg|\mathbb{E}_{(\boldsymbol{z},\boldsymbol{y})\sim \mathbb{P}} \Bigg\{\mathbbm{1}\big(\boldsymbol{z}\in Y(\boldsymbol{\phi},N_{\epsilon / \Psi})\big) \Bigg[  \sum_{\boldsymbol{x}_d\in \mathcal{X}_d}p(\boldsymbol{x}_d|M_{\boldsymbol{\phi}_n}(\boldsymbol{z}),\lambda_n) l((\boldsymbol{x}_d,\boldsymbol{x}_c^{\ast}(\boldsymbol{x}_d,M_{\boldsymbol{\phi}_n}(\boldsymbol{z}))),\boldsymbol{y})
        \\-
    l \bigg( \Big(\boldsymbol{x}_d^{\ast}(M_{\boldsymbol{\phi}}(\boldsymbol{z})),\boldsymbol{x}_c^{\ast}(\boldsymbol{x_d}^{\ast}(M_{\boldsymbol{\phi}}(\boldsymbol{z})),M_{\boldsymbol{\phi}}(\boldsymbol{z}))\Big),\boldsymbol{y}\bigg)\Bigg] \Bigg\}\Bigg|+ 2 \frac{\epsilon}{\Psi} \Psi
    \end{gathered}\label{thm1eq2}
\end{gather}
where $\mathbbm{1}$ is the indicator function.

In (\ref{thm1eq2}), for $\boldsymbol{x_d}^{\ast}(M_{\boldsymbol{\phi}}(\boldsymbol{z}))$,
\begin{gather}
    p\Big(\boldsymbol{x_d}^{\ast}(M_{\boldsymbol{\phi}}(\boldsymbol{z}))\Big|M_{\boldsymbol{\phi}_n}(\boldsymbol{z}),\lambda_n\Big)=\frac{\text{exp}\left( -\frac{f^{\ast}\Big(\boldsymbol{x_d}^{\ast}(M_{\boldsymbol{\phi}}(\boldsymbol{z})),M_{\boldsymbol{\phi}_n}(\boldsymbol{z})\Big)}{\lambda_n} \right) }{\sum_{\boldsymbol{x_d}^{'}\in \mathcal{X}_d}\text{exp}\left( -\frac{f^{\ast}\Big(\boldsymbol{x}_d^{'},M_{\boldsymbol{\phi}_n}(\boldsymbol{z})\Big)}{\lambda_n} \right)}\\
    =\frac{1}{1+\sum_{\boldsymbol{x_d}^{'}\in \mathcal{X}_d/\big\{ \boldsymbol{x_d}^{\ast}(M_{\boldsymbol{\phi}}(\boldsymbol{z})) \big\}}\text{exp}\left( -\frac{f^{\ast}\Big(\boldsymbol{x}_d^{'},M_{\boldsymbol{\phi}_n}(\boldsymbol{z})\Big)-f^{\ast}\Big(\boldsymbol{x_d}^{\ast}(M_{\boldsymbol{\phi}}(\boldsymbol{z})),M_{\boldsymbol{\phi}_n}(\boldsymbol{z})\Big)}{\lambda_n} \right)}
\end{gather}

Since $\boldsymbol{\phi}_n\to \boldsymbol{\phi}$, then for sufficiently large $n$, $|  \boldsymbol{\phi}_n- \boldsymbol{\phi} |\leq \frac{1}{N_{\epsilon/\Psi}}$. Therefore, for $\boldsymbol{z}\in{Y}(\boldsymbol{\phi},N_{\epsilon/\Psi})$,
\begin{equation}
    f^{\ast}\Big(\boldsymbol{x}_d^{'},M_{\boldsymbol{\phi}_n}(\boldsymbol{z})\Big)-f^{\ast}\Big(\boldsymbol{x_d}^{\ast}(M_{\boldsymbol{\phi}}(\boldsymbol{z})),M_{\boldsymbol{\phi}_n}(\boldsymbol{z})\Big)\geq \frac{1}{N_{\epsilon/\Psi}}, \forall \boldsymbol{x_d}^{'}\neq \boldsymbol{x_d}^{\ast}
\end{equation}

Therefore,
\begin{gather}
     p\Big(\boldsymbol{x_d}^{\ast}(M_{\boldsymbol{\phi}}(\boldsymbol{z}))\Big|M_{\boldsymbol{\phi}_n}(\boldsymbol{z}),\lambda_n\Big)\geq \frac{1}{1+\sum_{\boldsymbol{x_d}^{'}\in \mathcal{X}_d/\big\{ \boldsymbol{x_d}^{\ast}(M_{\boldsymbol{\phi}}(\boldsymbol{z})) \big\}}\text{exp}(-\frac{1}{ \lambda_nN_{\epsilon/\Psi}})}=\frac{1}{1+(|\mathcal{X}_d|-1)\text{exp}(-\frac{1}{ \lambda_nN_{\epsilon/\Psi}})}
\end{gather}

Since $\lambda_n \searrow 0^+$, we have 
\begin{equation}
   \lim_{n\to \infty}  p\Big(\boldsymbol{x_d}^{\ast}(M_{\boldsymbol{\phi}}(\boldsymbol{z}))\Big|M_{\boldsymbol{\phi}_n}(\boldsymbol{z}),\lambda_n\Big)=1
\end{equation}
So
\begin{equation}
    \lim_{n\to \infty} \sum_{\boldsymbol{x}_d\in \mathcal{X}_d}p(\boldsymbol{x}_d|M_{\boldsymbol{\phi}_n}(\boldsymbol{z}),\lambda_n) l((\boldsymbol{x}_d,\boldsymbol{x}_c^{\ast}(\boldsymbol{x}_d,M_{\boldsymbol{\phi}_n}(\boldsymbol{z}))),\boldsymbol{y})=
    l \bigg( \Big(\boldsymbol{x}_d^{\ast}(M_{\boldsymbol{\phi}}(\boldsymbol{z})),\boldsymbol{x}_c^{\ast}(\boldsymbol{x_d}^{\ast}(M_{\boldsymbol{\phi}}(\boldsymbol{z})),M_{\boldsymbol{\phi}}(\boldsymbol{z}))\Big),\boldsymbol{y}\bigg)
\end{equation}
Or equally
\begin{equation}
   \lim_{n\to \infty}\bigg\{  \sum_{\boldsymbol{x}_d\in \mathcal{X}_d}p(\boldsymbol{x}_d|M_{\boldsymbol{\phi}_n}(\boldsymbol{z}),\lambda_n) l((\boldsymbol{x}_d,\boldsymbol{x}_c^{\ast}(\boldsymbol{x}_d,M_{\boldsymbol{\phi}_n}(\boldsymbol{z}))),\boldsymbol{y})-
    l \bigg( \Big(\boldsymbol{x}_d^{\ast}(M_{\boldsymbol{\phi}}(\boldsymbol{z})),\boldsymbol{x}_c^{\ast}(\boldsymbol{x_d}^{\ast}(M_{\boldsymbol{\phi}}(\boldsymbol{z})),M_{\boldsymbol{\phi}}(\boldsymbol{z}))\Big),\boldsymbol{y}\bigg)\bigg\}=0
\end{equation}

Since $l$ is bounded, by the bounded convergence theorem, we have
\begin{equation}
    \begin{gathered}
    \lim_{n\to \infty}\Bigg|\mathbb{E}_{(\boldsymbol{z},\boldsymbol{y})\sim \mathbb{P}} \Bigg\{\mathbbm{1}\big(\boldsymbol{z}\in Y(\boldsymbol{\phi},N_{\epsilon / \Psi})\big) \Bigg[  \sum_{\boldsymbol{x}_d\in \mathcal{X}_d}p(\boldsymbol{x}_d|M_{\boldsymbol{\phi}_n}(\boldsymbol{z}),\lambda_n) l((\boldsymbol{x}_d,\boldsymbol{x}_c^{\ast}(\boldsymbol{x}_d,M_{\boldsymbol{\phi}_n}(\boldsymbol{z}))),\boldsymbol{y})
        \\-
    l \bigg( \Big(\boldsymbol{x}_d^{\ast}(M_{\boldsymbol{\phi}}(\boldsymbol{z})),\boldsymbol{x}_c^{\ast}(\boldsymbol{x_d}^{\ast}(M_{\boldsymbol{\phi}}(\boldsymbol{z})),M_{\boldsymbol{\phi}}(\boldsymbol{z}))\Big),\boldsymbol{y}\bigg)\Bigg] \Bigg\}\Bigg|=0
\end{gathered}
\end{equation}

Therefore, 
\begin{equation}
    \lim_{n\to \infty} |\mathcal{R}_{\lambda_n}(\boldsymbol{\phi}_n)- \mathcal{R}(\boldsymbol{\phi})|\leq 2\epsilon
\end{equation}

Since $\epsilon$ can be selected arbitrarily small, we have
\begin{equation}
    \lim_{n\to \infty} |\mathcal{R}_{\lambda_n}(\boldsymbol{\phi}_n)- \mathcal{R}(\boldsymbol{\phi})|=0
\end{equation}

Therefore, $\mathcal{R}_{\lambda_n}(\boldsymbol{\phi})$ epi-converges to $\mathcal{R}(\boldsymbol{\phi})$ as $n\to \infty$.

(ii) Since $\mathcal{R}_{\lambda_n}(\boldsymbol{\phi})$ epi-converges to $\mathcal{R}(\boldsymbol{\phi})$, (ii) can be immediately derived by applying Proposition 4.6 in \citet{bonnans2013perturbation}

\subsection{Proof of \cref{thm_diff_mixed}}
\label{proof_diff_mixed}
We first prove \cref{expect_grad}. By the chain rule,
\begin{equation}
   \frac{\partial r_{\lambda}(\boldsymbol{\theta},\boldsymbol{y})}{\partial \boldsymbol{\theta}}= \sum_{\boldsymbol{x_d}\in \mathcal{X}_d}\left\{ \frac{\partial p(\boldsymbol{x_d}|\boldsymbol{\theta},\lambda)}{\partial \boldsymbol{\theta}}l\Big(\big(\boldsymbol{x_d},\boldsymbol{x_c}^{\ast}(\boldsymbol{x_d},\boldsymbol{\theta})\big),\boldsymbol{y}\Big)+
    p(\boldsymbol{x_d}|\boldsymbol{\theta},\lambda)\frac{\partial l\Big(\big(\boldsymbol{x_d},\boldsymbol{x_c}^{\ast}(\boldsymbol{x_d},\boldsymbol{\theta})\big),\boldsymbol{y}\Big)}{\partial \boldsymbol{x_c}^{\ast}} \frac{\partial  \boldsymbol{x_c}^{\ast}(\boldsymbol{x_d},\boldsymbol{\theta})}{\partial \boldsymbol{\theta}}\right\}\label{proof4_1}
\end{equation}

By \cref{assum_mixed_int}, the feasible region $\mathcal{X}_c(\boldsymbol{x}_d)$ given $\boldsymbol{x}_d$ is a second-order cone representable set. Therefore, by \cref{diff_con}, the optimal continuous solution $\boldsymbol{x_c}^{\ast}(\boldsymbol{x_d},\boldsymbol{\theta})$ is differentiable with respect to $\boldsymbol{\theta}$, so the last gradient term $\frac{\partial  \boldsymbol{x_c}^{\ast}(\boldsymbol{x_d},\boldsymbol{\theta})}{\partial \boldsymbol{\theta}}$ in 
\cref{proof4_1} is well-defined.

Since the energy function 
\begin{equation}
    E(\boldsymbol{x}_d,\boldsymbol{\theta},\lambda)=\text{exp}\left( -\frac{f\big((\boldsymbol{x}_d,\boldsymbol{x}_c^{\ast}(\boldsymbol{x}_d,\boldsymbol{\theta})),\boldsymbol{\theta}\big)}{\lambda}  \right)
\end{equation}
and $\boldsymbol{x}_c^{\ast}(\boldsymbol{x}_d,\boldsymbol{\theta}))$ is differentiable, thus $E(\boldsymbol{x}_d,\boldsymbol{\theta},\lambda)$ is also differentiable with respect to $\boldsymbol{\theta}$.

Therefore, the first gradient term $\frac{\partial p(\boldsymbol{x_d}|\boldsymbol{\theta},\lambda)}{\partial \boldsymbol{\theta}}$ in \cref{proof4_1} is well-defined since
\begin{equation}
p(\boldsymbol{x}_d|\boldsymbol{\theta},\lambda)=\frac{E(\boldsymbol{x}_d,\boldsymbol{\theta},\lambda)}{\sum_{\boldsymbol{x}_d^{'}\in \mathcal{X}_d} E(\boldsymbol{x}_d^{'},\boldsymbol{\theta},\lambda)}
\end{equation}

Let $Z(\boldsymbol{\theta},\lambda)$ denote the normalizer $\sum_{\boldsymbol{x}_d^{'}\in \mathcal{X}_d} E(\boldsymbol{x}_d^{'},\boldsymbol{\theta},\lambda)$. By the chain rule, we have
\begin{gather}
    \frac{\partial p(\boldsymbol{x}_d|\boldsymbol{\theta},\lambda)}{\partial \boldsymbol{\theta}}
    =
    \frac{\partial \frac{E(\boldsymbol{x_d},\boldsymbol{\theta},\lambda)}{\sum_{\boldsymbol{x_d}^{'}\in \mathcal{X}_d}E(\boldsymbol{x_d}^{'},\boldsymbol{\theta},\lambda)}}{\partial \boldsymbol{\theta}}
    =
    \frac{E^{'} (\boldsymbol{x_d},\boldsymbol{\theta},\lambda)}{Z(\boldsymbol{\theta},\lambda)}-
    \frac{E (\boldsymbol{x_d},\boldsymbol{\theta},\lambda)}{Z(\boldsymbol{\theta},\lambda)} \frac{\sum_{\boldsymbol{x_d}^{'}} E^{'} (\boldsymbol{x_d}^{'},\boldsymbol{\theta},\lambda) }{Z(\boldsymbol{\theta},\lambda)}\\
    =
    p(\boldsymbol{x_d}|\boldsymbol{\theta},\lambda) \frac{E^{'} (\boldsymbol{x_d},\boldsymbol{\theta},\lambda)}{E (\boldsymbol{x_d},\boldsymbol{\theta},\lambda)}-p(\boldsymbol{x_d}|\boldsymbol{\theta},\lambda) \bigg[ \sum_{\boldsymbol{x_d}^{'}\in \mathcal{X}_d} p(\boldsymbol{x_d}^{'}|\boldsymbol{\theta},\lambda) \frac{E^{'} (\boldsymbol{x_d}^{'},\boldsymbol{\theta},\lambda)}{E (\boldsymbol{x_d}^{'},\boldsymbol{\theta},\lambda)} \bigg]\\
    =p(\boldsymbol{x_d}|\boldsymbol{\theta},\lambda)\Bigg\{ \frac{E^{'} (\boldsymbol{x_d},\boldsymbol{\theta},\lambda)}{E (\boldsymbol{x_d},\boldsymbol{\theta},\lambda)}-\mathbb{E}_{\boldsymbol{x_d}^{'}\sim p(\boldsymbol{x_d}^{'}|\boldsymbol{\theta},\lambda)}\bigg[  \frac{E^{'} (\boldsymbol{x_d}^{'},\boldsymbol{\theta},\lambda)}{E (\boldsymbol{x_d}^{'},\boldsymbol{\theta},\lambda)}  \bigg]\Bigg\}\label{proof4_2}
\end{gather}

By combining \cref{proof4_1} and \cref{proof4_2}, we have
\begin{gather}
\begin{aligned}
    \frac{\partial r_{\lambda}(\boldsymbol{\theta},\boldsymbol{y})}{\partial \boldsymbol{\theta}}&=\sum_{\boldsymbol{x_d}\in \mathcal{X}_d} p(\boldsymbol{x_d}|\boldsymbol{\theta},\lambda)\frac{E^{'} (\boldsymbol{x_d}^{'},\boldsymbol{\theta},\lambda)}{E (\boldsymbol{x_d}^{'},\boldsymbol{\theta},\lambda)}l\Big(\big(\boldsymbol{x_d},\boldsymbol{x_c}^{\ast}(\boldsymbol{x_d},\boldsymbol{\theta})\big),\boldsymbol{y}\Big) 
    \\&-\mathbb{E}_{\boldsymbol{x_d}^{'}\sim p(\boldsymbol{x_d}^{'}|\boldsymbol{\theta},\lambda)}\bigg[  \frac{E^{'} (\boldsymbol{x_d}^{'},\boldsymbol{\theta},\lambda)}{E (\boldsymbol{x_d}^{'},\boldsymbol{\theta},\lambda)}\bigg] \Bigg( \sum_{\boldsymbol{x_d}\in \mathcal{X}_d} p(\boldsymbol{x_d}|\boldsymbol{\theta},\lambda)l\Big(\big(\boldsymbol{x_d},\boldsymbol{x_c}^{\ast}(\boldsymbol{x_d},\boldsymbol{\theta})\big),\boldsymbol{y}\Big)  \Bigg)\\
    &+\sum_{\boldsymbol{x_d}\in \mathcal{X}_d}
    p(\boldsymbol{x_d}|\boldsymbol{\theta},\lambda)\frac{\partial l\Big(\big(\boldsymbol{x_d},\boldsymbol{x_c}^{\ast}(\boldsymbol{x_d},\boldsymbol{\theta})\big),\boldsymbol{y}\Big)}{\partial \boldsymbol{x_c}^{\ast}} \frac{\partial  \boldsymbol{x_c}^{\ast}(\boldsymbol{x_d},\boldsymbol{\theta})}{\partial \boldsymbol{\theta}}\end{aligned}\\
    \begin{aligned}
            &=
            \mathbb{E}_{\boldsymbol{x_d}\sim p(\boldsymbol{x_d}|\boldsymbol{\theta},\lambda)}  \Bigg[ \frac{E^{'} (\boldsymbol{x_d},\boldsymbol{\theta},\lambda)}{E (\boldsymbol{x_d},\boldsymbol{\theta},\lambda)}l\Big(\big(\boldsymbol{x_d},\boldsymbol{x_c}^{\ast}(\boldsymbol{x_d},\boldsymbol{\theta})\big),\boldsymbol{y}\Big)\Bigg]\\&-\mathbb{E}_{\boldsymbol{x_d}\sim p(\boldsymbol{x_d}|\boldsymbol{\theta},\lambda)}\bigg[  \frac{E^{'} (\boldsymbol{x_d},\boldsymbol{\theta},\lambda)}{E (\boldsymbol{x_d},\boldsymbol{\theta},\lambda)}  \bigg]  \mathbb{E}_{\boldsymbol{x_d}\sim p(\boldsymbol{x_d}|\boldsymbol{\theta},\lambda)}\bigg[ l\Big(\big(\boldsymbol{x_d},\boldsymbol{x_c}^{\ast}(\boldsymbol{x_d},\boldsymbol{\theta})\big),\boldsymbol{y}\Big)  \bigg]\\&
    +\mathbb{E}_{\boldsymbol{x_d}\sim p(\boldsymbol{x_d}|\boldsymbol{\theta},\lambda)}\Bigg[ \frac{\partial l\Big(\big(\boldsymbol{x_d},\boldsymbol{x_c}^{\ast}(\boldsymbol{x_d},\boldsymbol{\theta})\big),\boldsymbol{y}\Big)}{\partial \boldsymbol{x_c}^{\ast}} \frac{\partial  \boldsymbol{x_c}^{\ast}(\boldsymbol{x_d},\boldsymbol{\theta})}{\partial \boldsymbol{\theta}} \Bigg]
        \end{aligned}
\end{gather}
Therefore, we have derived \cref{expect_grad}.

Since $r_{\lambda}(\boldsymbol{\theta},\boldsymbol{y})=r_{\lambda}(M_{\boldsymbol{\phi}}(\boldsymbol{z}),\boldsymbol{y})$ is bounded by \cref{assum_bound}, then according to Theorem 9.56 in \citet{shapiro2021lectures}, 
\begin{equation}
    \frac{\partial \mathcal{R}_{\lambda}(\boldsymbol{\phi})}{\partial \boldsymbol{\phi}} =\frac{\partial \mathbb{E}_{(\boldsymbol{z},\boldsymbol{y})\sim \mathbb{P}} \ r_{\lambda}(M_{\boldsymbol{\phi}}(\boldsymbol{z}),\boldsymbol{y})}{\partial \boldsymbol{\phi}}=\mathbb{E}_{(\boldsymbol{z},\boldsymbol{y})\sim \mathbb{P}} \bigg[ 
    \frac{\partial r_{\lambda}(M_{\boldsymbol{\phi}}(\boldsymbol{z}),\boldsymbol{y})}{\partial \boldsymbol{\phi}}\bigg]=\mathbb{E}_{(\boldsymbol{z},\boldsymbol{y})\sim \mathbb{P}}\bigg[ \frac{\partial r_{\lambda}(\boldsymbol{\theta},\boldsymbol{y})}{\partial \boldsymbol{\theta}} \frac{\partial \boldsymbol{\theta}}{\partial \boldsymbol{\phi}}\bigg]
\end{equation}

\section{Extension to Wasserstein DRO Layer}
\label{wass_extend}

Here we present how to extend the proposed DRO Layer to Wasserstein-based DRO with a learnable radius.

In the non-contextual setting, the reference distribution of the Wasserstein ambiguity set is typically set to an empirical distribution of $N$ data points. In the contextual setting, we take the idea in \citet{bertsimas2020predictive} and set the conditional empirical distribution $\widehat{\mathbb{P}}(\boldsymbol{z})$ as a weighted sum of data points, \emph{i.e.}
\begin{equation}
    \widehat{\mathbb{P}}(\boldsymbol{z})=\sum_{n=1}^N \omega_n(\boldsymbol{z}) \delta_{\boldsymbol{y}_n}
\end{equation}
where $\delta_{[\cdot]}$ is the Dirac delta function and $\omega_n(\boldsymbol{z}),n\in[N]$ are the weights satisfying $\sum_{n\in [N]}\omega_n(\boldsymbol{z})=1$.

Let $\epsilon_{\theta}(\boldsymbol{z})$ be the learnable radius with parameter $\theta$. Then, the Wasserstein ambiguity set is 
\begin{equation}
    \mathscr{U}\Big(\widehat{\mathbb{P}}(\boldsymbol{z}),\epsilon_{\theta} (\boldsymbol{z})\Big)=\left\{\mathbb{P} \left|  \mathbb{P}(\varXi)=1, d_W\left( \widehat{\mathbb{P}}(\boldsymbol{z}),\mathbb{P}\right)\leq \epsilon_{\theta} (\boldsymbol{z}) \right.\right\},
\end{equation}
and the Wasserstein DRO problem is
\begin{equation}\tag{WDRO}\label{WDRO}
    \min_{\boldsymbol{x}\in \mathcal{X}}\max_{\mathbb{P}\in \mathscr{U}\left(\widehat{\mathbb{P}}(\boldsymbol{z}),\epsilon_{\theta} (\boldsymbol{z})\right)} \mathbb{E}_{\mathbb{P}}[c(\boldsymbol{x},\boldsymbol{y})]
\end{equation}
where $\boldsymbol{x}$ is the mixed-integer decision variable satisfying \cref{assum_mixed_int} and $c(\boldsymbol{x},\boldsymbol{y})$ is the cost function  satisfying \cref{ass_cost_func}.

As outlined in the paper, the procedure of the proposed DRO Layer is presented as follows.
\begin{gather*}
    \boldsymbol{z} \stackrel{\theta}{\longrightarrow}  \mathscr{U}\Big( \widehat{\mathbb{P}}(\boldsymbol{z}),\epsilon_{\theta} (\boldsymbol{z}) \Big) \longrightarrow \text{Problem (WDRO)}  \xrightarrow{\text{Solve (WDRO) for } T \text{ times}} \text{Construct proposal distribution}\\ \xrightarrow{\text{Importance sampling } } \text{Compute gradient (\ref{expect_grad})}\to \theta \text{ update}
\end{gather*}

Therefore, in order to learn the ambiguity set in a decision-focused style, we only need to ensure that 
\begin{enumerate}[(i)]
    \item The problem (\ref{WDRO}) is a mixed-integer linear cone programming.\label{conn1}
    \item When the integer part of the decision variable (\emph{i.e.}, $\boldsymbol{x}_d$) is fixed, the problem (\ref{WDRO}) is a linear cone programming.\label{conn2}
\end{enumerate}
where condition (\ref{conn1}) allows us to use a commercial solver like Gurobi to solve problem (\ref{WDRO}) for $T$ times so that we can construct the proposal distribution, and condition (\ref{conn2}) allows us to derive the gradient of continuous variables with respect to learnable parameter in computing gradient (\ref{expect_grad}).

In fact, these two conditions are satisfied for Wasserstein DRO, and we formally present this result in the following corollary.
\begin{corollary}
    If the mixed-integer decision variable $\boldsymbol{x}$ satisfies Assumption 4.7 and 
 the cost function $c(\boldsymbol{x},\boldsymbol{y})$ satisfies Assumption 4.8, then condition (\ref{conn1}) and (\ref{conn2}) hold for the problem (\ref{WDRO}).
\end{corollary}
\begin{proof}
    The proof is straightforward by combining techniques in Wasserstein DRO \cite{mohajerin2018data} and second-order cone programming \cite{ben2001lectures}.

    In fact, since Assumption 4.7 holds, we only need to show that condition (\ref{conn2}) holds when the decision variable $\boldsymbol{x}$ is pure continuous and the feasible region $\mathcal{X}$ is second-order cone representable. For simplicity, we prove the corollary for bilinear cost function
    \begin{equation}
        c(\boldsymbol{x},\boldsymbol{y})=\max_{k\in[K]} \boldsymbol{x}^T\boldsymbol{T}_k\boldsymbol{y}
    \end{equation}
    and proof for general cost functions satisfying \cref{ass_cost_func} is quite similar but with heavier notations.

    According to \citet{mohajerin2018data}, the problem (\ref{WDRO}) can be reformulated as
    \begin{equation}\label{refor_wass}
        \begin{gathered}
        \inf_{\boldsymbol{x}\in \mathcal{X},\lambda,s_n,\boldsymbol{\gamma}_{nk}} \lambda \epsilon_{\theta} (\boldsymbol{z}) + \sum_{n=1}^N \omega_n(\boldsymbol{z}) s_n\\
        \text{s.t. } s_n\geq\sup_{\boldsymbol{y}\in \varXi} \left( \boldsymbol{x}^T \boldsymbol{T}_k \boldsymbol{y} - \boldsymbol{\gamma}_{nk}^T \boldsymbol{y} \right)+\boldsymbol{\gamma}_{nk}^T \boldsymbol{y}_n,\ \forall n\in [N],\forall k\in [K]\\
        \|  \boldsymbol{\gamma}_{nk} \|_{\ast}\leq \lambda,\ \forall n\in [N],\forall k\in [K]
    \end{gathered}
    \end{equation}
    Since the uncertainty support $\varXi$ is a second-order cone representable set, it can be formulated as follows.
    \begin{equation}\label{support}
        \varXi=\left\{\boldsymbol{y}\ |\ \exists \boldsymbol{v} \text{ s.t. }  \boldsymbol{A}_j^{\boldsymbol{y}} \boldsymbol{y}+\boldsymbol{A}_j^{\boldsymbol{v}}-\boldsymbol{b}_j\geq_{L^{m_j}} \boldsymbol{0} ,\forall j\in [J]   \right\}
    \end{equation}

    By leveraging expression (\ref{support}), we can further reformulate (\ref{refor_wass}) as 
    \begin{equation}\label{refor_1}
        \begin{gathered}
        \inf_{\boldsymbol{x}\in \mathcal{X},\lambda,s_n,\boldsymbol{\gamma}_{nk},\boldsymbol{\eta}_{nkj}} \lambda \epsilon_{\theta} (\boldsymbol{z}) + \sum_{n=1}^N \omega_n(\boldsymbol{z}) s_n\\
        \text{s.t. } s_n\geq \sum_{j\in [J] } -\boldsymbol{b}_j^T \boldsymbol{\eta}_{nkj},\ \forall n\in [N],\forall k\in [K]\\
        \boldsymbol{T}_k^T\boldsymbol{x}-\boldsymbol{\gamma}_{nk}+\sum_{j\in [J]} {\boldsymbol{A}_j^{\boldsymbol{y}}}^T\boldsymbol{\eta}_{nkj}=\boldsymbol{0},\ \forall n\in [N],\forall k\in [K] \\
        \sum_{j\in [J] }{\boldsymbol{A}_j^{\boldsymbol{v}}}^T\boldsymbol{\eta}_{nkj}=\boldsymbol{0},\ \forall n\in [N],\forall k\in [K] \\
        \boldsymbol{\eta}_{nkj}\geq_{L^{m_j}}\boldsymbol{0},\ \forall n\in [N],\forall k\in [K], \forall j \in [J]\\
        \|  \boldsymbol{\gamma}_{nk} \|_{\ast}\leq \lambda,\ \forall n\in [N],\forall k\in [K]
    \end{gathered}
    \end{equation}
    Problem (\ref{refor_1}) is already in the form of a linear cone programming.
\end{proof}

\section{Experiment Setup}
\label{exp_set}

\subsection{Toy Example: Multi-item Newsvendor Problem}
\label{more_news}

We measure the performance by the following percentage optimality gap.
\begin{equation}
    \text{Percentage Optimality Gap}=\frac{\mathbb{E}_{(\boldsymbol{z},\boldsymbol{y})\sim \mathbb{P}}\ l( \boldsymbol{x}^{\text{learning}}(\boldsymbol{z} ) ,\boldsymbol{y}) -\mathbb{E}_{(\boldsymbol{z},\boldsymbol{y})\sim \mathbb{P}}\ l( \boldsymbol{x}^{\ast}(\boldsymbol{z} ) ,\boldsymbol{y}) }{\mathbb{E}_{(\boldsymbol{z},\boldsymbol{y})\sim \mathbb{P}}\ l( \boldsymbol{x}^{\ast}(\boldsymbol{z} ) ,\boldsymbol{y}) }
\end{equation}
where $\boldsymbol{x}^{\text{learning}}(\boldsymbol{z} ) $ is the decision given by the learning method and $ \boldsymbol{x}^{\ast}(\boldsymbol{z} )$ is the optimal decision derived by solving the following problem.
\begin{equation}
    \boldsymbol{x}^{\ast}(\boldsymbol{z} )=\argmin_{\boldsymbol{x}\in \mathcal{X}} \mathbb{E}_{\boldsymbol{y}\sim p(\boldsymbol{y}|\boldsymbol{z})}l(\boldsymbol{x},\boldsymbol{y})
\end{equation}

In computing the percentage optimality gap, we compute $\mathbb{E}_{(\boldsymbol{z},\boldsymbol{y})\sim \mathbb{P}}\ l( \boldsymbol{x}^{\text{learning}}(\boldsymbol{z} ) ,\boldsymbol{y})$ by sample average approximation using $1\times 10^5$ pairs of $(\boldsymbol{z},\boldsymbol{y})$ i.i.d. sampled from $\mathbb{P}$. 
Since 
\begin{equation}\label{per_saa}
    \mathbb{E}_{(\boldsymbol{z},\boldsymbol{y})\sim \mathbb{P}}\ l( \boldsymbol{x}^{\ast}(\boldsymbol{z} ) ,\boldsymbol{y})=\mathbb{E}_{\boldsymbol{z}\sim \mathbb{P}_{\boldsymbol{z}}} \bigg[ \min_{\boldsymbol{x}\in \mathcal{X}}\mathbb{E}_{\boldsymbol{y}\sim p(\boldsymbol{y}|\boldsymbol{z})}\ l( \boldsymbol{x},\boldsymbol{y}) \bigg]
\end{equation}
we follow a two step approach to compute $\mathbb{E}_{(\boldsymbol{z},\boldsymbol{y})\sim \mathbb{P}}\ l( \boldsymbol{x}^{\ast}(\boldsymbol{z} ) ,\boldsymbol{y})$. The outer expectation in \cref{per_saa} is approximated by sample average approximation using 400 $\boldsymbol{z}$ i.i.d. sampled from the marginal distribution $\mathbb{P}_{\boldsymbol{z}}$. The inner stochastic program $\min_{\boldsymbol{x}\in \mathcal{X}}\mathbb{E}_{\boldsymbol{y}\sim p(\boldsymbol{y}|\boldsymbol{z})}\ l( \boldsymbol{x},\boldsymbol{y})$ is also solved by sample average approximation using 200 $\boldsymbol{y}$ i.i.d. sampled from the conditional distribution $p(\boldsymbol{y}|\boldsymbol{z})$.

In experiments, we use \emph{k}-nearest-neighbors weight function to construct the distribution $\mathbb{Q}_{\boldsymbol{z}}$ (defined in \cref{W_SAA}) in the projection layer, \emph{i.e.},
\begin{equation}\label{nominal_distri}
    \omega_n(\boldsymbol{z})=\frac{1}{K}, \ \text{if }\boldsymbol{z}_n \text{ is in the k-nearest-neighbor of }\boldsymbol{z}. \text{ Else, }
        \omega_n(\boldsymbol{z})=0.
\end{equation}

In the experiments, we select $K$ as $\frac{N}{20}$, where $N$ is the training data size.

We use neural networks to learn the ambiguity set parameters $\boldsymbol{\mu}_\text{I},\sigma_\text{I},\boldsymbol{\mu}_\text{II},\sigma_\text{II}$, the architecture are presented as follows.
\begin{gather}
    \boldsymbol{\mu}_\text{I}:\ \text{FC}(1,60)\to \text{FC}(60,60)\to \text{FC}(60,4)\\
    {\sigma}_\text{I}:\ \text{FC}(1,60)\to \text{FC}(60,60)\to \text{FC}(60,1)\\
    \boldsymbol{\mu}_\text{II}:\ \text{FC}(1,60)\to \text{FC}(60,60)\to \text{FC}(60,4)\\
    {\sigma}_\text{II}:\ \text{FC}(1,60)\to \text{FC}(60,60)\to \text{FC}(60,1)
\end{gather}
where FC$(m_1,m_2)$ represents full connection layer with $m_1$ inputs and $m_2$ outputs.

Since $\boldsymbol{\mu}$ and $\sigma$ are learned by different neural networks, in the pre-training, we first train the parameter $\boldsymbol{\mu}$ by minimizing $\Big\|   \mathbb{E}_{ \mathbb{Q}_{\boldsymbol{z}}}[g_i(\boldsymbol{y},\boldsymbol{\alpha}_i)]     \Big\|$ and then train $\sigma$ by minimizing $\Big\|   \mathbb{E}_{ \mathbb{Q}_{\boldsymbol{z}}}[g_i(\boldsymbol{y},\boldsymbol{\alpha}_i)]-\sigma_i     \Big\|$.

In conducting experiments, we vary the training data size $N$ from 100 to 800, and in each case, 10 runs are conducted. In each case, the size of the validation data set is set to $\frac{N}{5}$. In generating data, we first sample covariate $\boldsymbol{z}$, and then sample $\boldsymbol{y}$ conditioned on $\boldsymbol{z}$.

In the multi-item newsvendor problem (\ref{news_pro}) with $n=4$, we set $\boldsymbol{a}^c=(0.25,0.5,0.75,1), \boldsymbol{a}^d=0.95\times\boldsymbol{a}^c,\boldsymbol{v}=(10.0,13.0,16.0,19.0),\boldsymbol{b}=(2.0,4.0,6.0,8.0),\boldsymbol{d}=(0.5,1.0,1.5,2.0)$. The covariate $\boldsymbol{z}$ follows the uniform distribution on $[0,1]$, and conditioned on $\boldsymbol{z}$, we set the distribution of demand $\boldsymbol{y}$ by
\begin{gather*}
    y_1\sim \mathcal{N}\left( 8+6(z-0.2)^2,\frac{1}{1+8|z-0.2|} \right),
    y_2\sim \mathcal{N}\left( 11+6(z-0.4)^2,\frac{1}{1+8|z-0.4|} \right)\\
    y_3\sim \mathcal{N}\left( 14+6(z-0.6)^2,\frac{1}{1+8|z-0.6|} \right),
    y_4\sim \mathcal{N}\left( 17+6(z-0.8)^2,\frac{1}{1+8|z-0.8|} \right)
\end{gather*}

\subsection{Portfolio Management Problem}
\label{more_prot}

In the experiment on the portfolio management problem, the covariate $\boldsymbol{z}$ is a vector of 5 dimensions, and the returns of $n$ assets are generated by the following non-linear model.
\begin{equation}
    \boldsymbol{y}= \boldsymbol{a} + \boldsymbol{B} \boldsymbol{z} + \boldsymbol{C} \boldsymbol{e} + \boldsymbol{D} \text{flat}(\boldsymbol{z}\boldsymbol{z}^T) + \| \boldsymbol{z} \|_1 \text{diag}(\boldsymbol{s}) \boldsymbol{g}
\end{equation}
where matrixes $\boldsymbol{B}\in \mathbb{R}^{n\times 5}, \boldsymbol{C}\in \mathbb{R}^{n\times 3}, \boldsymbol{D}\in \mathbb{R}^{n\times 25}$ and vector $ \boldsymbol{s}\in\mathbb{R}^{n}$ are randomly picked parameters, $\boldsymbol{e}\in\mathbb{R}^{3}$ and $ \boldsymbol{g}\in\mathbb{R}^{n}$ are random variables independent of covariate $\boldsymbol{z}\in\mathbb{R}^{5}$, and $\| \cdot \|_1$ is the 1-norm.

We use 2,500 data for training, 500 data for validation, and 1,000 data for testing.
The distribution $\mathbb{Q}_{\boldsymbol{z}}$ is also set to k-nearest distribution as in (\ref{nominal_distri}), and $K$ is set to 20.

In the gradient estimation, we solve the DRO 3 times to construct (\ref{impr_asmp}) and use 4 samples to estimate the gradient term by importance sampling. For the energy parameter $\lambda$ in (\ref{eneger_defnm}), we initially set it to 10, and subsequently reduce it by one-third every 30 epochs.

\subsubsection{Portfolio Management Problem with Pure Continuous Decisions}
We use the following two-layer full-connected neural network to learn the ambiguity set parameter in both our method and the method proposed in \citet{costa2023distributionally}.
\begin{gather}
    \text{FC}(5,22)\to \text{FC}(22,27)\to \text{FC}(27,m)
\end{gather}
where $m$ is the dimension of the ambiguity set parameter.

\subsubsection{Portfolio Management Problem with Mixed-Integer Decisions}

In problem (\ref{port_prob_mix}), the number of binary decisions is set to $\frac{n}{5}$ where $n$ is the number of assets, and the problem parameter $\boldsymbol{v}\in \mathbb{R}^{n/5}$ is set to $[1/n,\cdots,1/n]$.

To analyze the impact of the number of constraints in the SOC ambiguity set on the performance of the proposed decision-focused learning, we conduct experiments on the 60-dimensional mixed-integer portfolio management problem and consider the following three types of SOC ambiguity sets.
\begin{gather}
    \mathscr{U}_{15}(\boldsymbol{\mu},\boldsymbol{\sigma})=\left\{ \mathbb{P} \left|\ \begin{gathered}
        \mathbb{P}(\varXi)=1\\
            \mathbb{E}_{\mathbb{P}} \left[\sum_{j=1}^4 \left\| \boldsymbol{y}_{4(i-1)+j}-\boldsymbol{\mu}_{4(i-1)+j}  \right\|_2^2 \right]\leq \boldsymbol{\sigma}_i,\forall  i \in [15]
        \end{gathered}
         \right.       \right\}\\
         \mathscr{U}_{30}(\boldsymbol{\mu},\boldsymbol{\sigma})=\left\{ \mathbb{P} \left|\ \begin{gathered}
        \mathbb{P}(\varXi)=1\\
            \mathbb{E}_{\mathbb{P}} \left[\sum_{j=1}^2 \left\| \boldsymbol{y}_{2(i-1)+j}-\boldsymbol{\mu}_{2(i-1)+j}  \right\|_2^2 \right]\leq \boldsymbol{\sigma}_i,\forall  i \in [30]
        \end{gathered}
         \right.       \right\}\\
         \mathscr{U}_{60}(\boldsymbol{\mu},\boldsymbol{\sigma})=\left\{ \mathbb{P} \left|\ \begin{gathered}
        \mathbb{P}(\varXi)=1\\
            \mathbb{E}_{\mathbb{P}} \left[ \left\| \boldsymbol{y}_{i}-\boldsymbol{\mu}_{i}  \right\|_2^2 \right]\leq \boldsymbol{\sigma}_i,\forall  i \in [60]
        \end{gathered}
         \right.       \right\}
\end{gather}

Intuitively, $\mathscr{U}_{15}$ constrains 4 dimensions together and thus leads to 15 constraints, $\mathscr{U}_{30}$ constrains 2 dimensions together and thus leads to 30 constraints, and $\mathscr{U}_{60}$ constrains different dimensions individually and leads to 60 constraints.

\section{Discussion of Limitations}
\label{diss_limm}
The decision-focused learning pipeline we developed in Section 4 can be decomposed into four processes: 1. learning layer; 2. projection layer; 3. solving MICP; 4. DRO Layer.

Processes 2 and 4 are built on Cvxpylayers \cite{agrawal2019differentiable}, and Process 3 is built on commercial solver Gurobi. 
Processes 1, 2, and 4 participate in both the forward and backward path, and solving MICP is only involved in the forward path.

To test the scalability of our method, we test the running time for a batch of 100 instances on large-scale problems, and the results are presented in \cref{large_scale}. The experiments are conducted on a laptop with an i7 CPU and 32G RAM. In \cref{large_scale}, the running time for solving MICP and DRO Layer pertains to $T=1$ and $S=1$, so the total computational time for these two processes should be multiplied by $T$ and $S$, respectively.

From \cref{large_scale}, we can see the computational time is very long in high-dimensional problems. However, we note that this computational inefficiency is due to the inefficiency of Cvxpylayers and Gurobi, because both of them are built on \textbf{CPU} rather than on GPU.

\begin{table*}[htpb!]
\centering
\caption{Running time for a batch of 100 instances on problems with different dimensions. ($T$ stands for the number of solutions we derive to construct the proposal distribution, and $S$ stands for the number of sampling to estimate gradient (\ref{expect_grad}).)}
\begin{tabular}{ccccccc} 
\toprule
\multirow{2}*{Dimension} & \multicolumn{4}{c}{Forward path}  & \multirow{2}*{Backward propagation}  \\
\cmidrule(lr){2-5}  
   &   learning layer & projection layer & solving MICP ($\times T$) & DRO Layer ($\times S$)  \\
  \midrule
  100 & $<1$ms & 16.5s & 8.4s & 1.2s & 9.8s\\
  200 & $<1$ms & 19.6s & 12.8s& 1.4s& 13.1s\\
  400 & $<1$ms & 44.0s &24.1s &5.8s &90.5s\\
  800 & $<1$ms & 82.9s &47.6s &15.3s &358.7s\\
  \bottomrule
  \end{tabular}
  \label{large_scale}
\end{table*}
\end{document}